\documentclass[11pt]{article}
\usepackage{latexsym,amsfonts,amssymb,amsmath,amsthm}
\usepackage{graphicx}

\usepackage[usenames,dvipsnames]{color}
\usepackage{ulem}
\usepackage{xcolor}

\begin{document}
\setlength{\baselineskip}{16pt}

\parindent 0.5cm
\evensidemargin 0cm \oddsidemargin 0cm \topmargin 0cm \textheight 22.5cm \textwidth 16cm \footskip 2cm \headsep
0cm

\newtheorem{theorem}{Theorem}[section]
\newtheorem{lemma}{Lemma}[section]
\newtheorem{proposition}{Proposition}[section]
\newtheorem{definition}{Definition}[section]
\newtheorem{example}{Example}[section]
\newtheorem{corollary}{Corollary}[section]

\newtheorem{remark}{Remark}[section]

\numberwithin{equation}{section}

\def\p{\partial}
\def\I{\textit}
\def\R{\mathbb R}
\def\C{\mathbb C}
\def\u{\underline}
\def\l{\lambda}
\def\a{\alpha}
\def\O{\Omega}
\def\e{\epsilon}
\def\ls{\lambda^*}
\def\D{\displaystyle}
\def\wyx{ \frac{w(y,t)}{w(x,t)}}
\def\imp{\Rightarrow}
\def\tE{\tilde E}
\def\tX{\tilde X}
\def\tH{\tilde H}
\def\tu{\tilde u}
\def\d{\mathcal D}
\def\aa{\mathcal A}
\def\DH{\mathcal D(\tH)}
\def\bE{\bar E}
\def\bH{\bar H}
\def\M{\mathcal M}

\def\disp{\displaystyle}
\def\undertex#1{$\underline{\hbox{#1}}$}
\def\card{\mathop{\hbox{card}}}
\def\sgn{\mathop{\hbox{sgn}}}
\def\exp{\mathop{\hbox{exp}}}
\def\OFP{(\Omega,{\cal F},\PP)}
\newcommand\JM{Mierczy\'nski}

\newcommand\Q{\ensuremath{\mathbb{Q}}}
\newcommand\Z{\ensuremath{\mathbb{Z}}}
\newcommand\N{\ensuremath{\mathbb{N}}}

\title{Two-species chemotaxis-competition
system with singular sensitivity: Global existence, boundedness, and persistence}
\author{Halil Ibrahim Kurt and Wenxian Shen   \\
Department of Mathematics and Statistics\\
Auburn University\\
Auburn University, AL 36849\\
U.S.A. }

\date{}
\maketitle

\begin{abstract}
This paper is concerned with the following parabolic-parabolic-elliptic chemotaxis system with singular sensitivity and Lotka-Volterra competitive kinetics,
 \begin{equation}
\label{abstract-eq}
\begin{cases}
u_t=\Delta u-\chi_1 \nabla\cdot (\frac{u}{w} \nabla w)+u(a_1-b_1u-c_1v) ,\quad &x\in \Omega\cr
v_t=\Delta v-\chi_2 \nabla\cdot (\frac{v}{w} \nabla w)+v(a_2-b_2v-c_2u),\quad &x\in \Omega\cr
0=\Delta w-\mu w +\nu u+ \lambda v,\quad &x\in \Omega \cr
\frac{\p u}{\p n}=\frac{\p v}{\p n}=\frac{\p w}{\p n}=0,\quad &x\in\p\Omega,
\end{cases}
\end{equation}
where $\Omega \subset \mathbb{R}^N$ is a bounded smooth  domain, and   $\chi_i$, $a_i$, $b_i$, $ c_i$ ($i=1,2$)  and
$\mu,\, \nu, \, \lambda$
 are positive constants. This is the first work on two-species chemotaxis-competition system with singular
sensitivity and  Lotka-Volterra competitive kinetics.
Among others, we prove that  for any given nonnegative initial data $u_0,v_0\in C^0(\bar\Omega)$ with $u_0+v_0\not \equiv 0$, \eqref{abstract-eq} has
a unique globally defined classical solution   $(u(t,x;u_0,v_0),v(t,x;u_0,v_0),w(t,x;u_0,v_0))$ with
$u(0,x;u_0,v_0)=u_0(x)$ and $v(0,x;u_0,v_0)=v_0(x)$ provided that $\min\{a_1,a_2\}$ is large relative to $\chi_1,\chi_2$ and $u_0+v_0$ is not small. Moreover, under the same condition, we prove that
\begin{equation*}
    \limsup_{t\to\infty}
\|u(t,\cdot;u_0,v_0)+v(t,\cdot;u_0,v_0)\|_\infty\le M^*,
\end{equation*}
and
\begin{equation*}
    \liminf_{t\to\infty} \inf_{x\in\Omega}(u(t,x,u_0,v_0)+v(t,x;u_0,v_0))\ge m^*,
\end{equation*}
for some positive constants $M^*,m^*$  independent of $u_0,v_0$,
the latter is referred to as  combined pointwise persistence.

\end{abstract}


\medskip


\section{Introduction and Main Results}
\label{S:intro}

Chemotaxis refers to the movement of cells or organisms in response to chemicals in their environments, and plays a crucial role in many biological processes  such as immune system response, tumor growth, population dynamics, gravitational collapse, the governing of immune cell migration. Since the pioneering works by Keller and Segel (\cite{ke-se1}, \cite{Keller-1}) on chemotaxis models, a lot of works have been {carried out} on the qualitative properties of various chemotaxis models such as the analysis of global existence, boundedness, blow-up in finite time, and asymptotic behavior of globally defined   solutions, etc.  The reader is referred to \cite{bel-wi,  hillen-painter,Hor} and the references therein for some detailed introduction
into the mathematics of chemotaxis models.

There are a large number of works on various two competing species chemotaxis models. For example,
consider the following two-species chemotaxis system,
\begin{equation}
\label{main-eq0}
\begin{cases}
u_t= \Delta u-  \nabla\cdot (u \chi_1(w) \nabla w)+u(a_1-b_1u-c_1v) ,\quad &x\in \Omega\cr
v_t= \Delta v- \nabla\cdot (v\chi_2(w)\nabla w)+ v(a_2 -b_2v- c_2u),\quad &x\in \Omega\cr
\tau w_t= \Delta w-{\mu w}  +\nu u+ \lambda  v,\quad &x\in \Omega\cr
\frac{\p u}{\p n}=\frac{\p v}{\p n}=\frac{\p w}{\p n}=0,\quad &x\in\p\Omega,
\end{cases}\,
\end{equation}
where $\Omega\subset\R^N$ is a bounded smooth domain,
$a_1,b_1,c_1,a_2,b_2,c_2, \nu,\lambda$ are positive numbers, and $\tau\ge 0$.  Biologically, \eqref{main-eq0} models the evolution of two competitive species subject to a chemical substance, which is produced by the two species themselves.
Here the unknown functions $u(t,x)$ and $v(t,x)$ represent the population densities
of two competitive   biological species and  $w(t,x)$  represents  the concentration of
the chemical substance.  The terms $u(a_1-b_1u-c_1v)$ and
$v(a_2-b_2v-c_2u)$ are referred to as Lotka-Volterra competitive terms.  The parameter $\mu$ is the degradation rate of the chemical substance and $\nu$ and $\lambda$ are the production rates of the chemical substance by the species $u$ and $v$, respectively.  $\tau\ge 0$ is related to the diffusion rate of the chemical substance.  The functions  $\chi_1(w)$ and $\chi_2(w)$ reflect the strength of the chemical substance on the movements of two species, and are referred to as chemotaxis sensitivity functions or coefficients.

When $\chi_1(w)\equiv \chi_1 >0$, $\chi_2(w)\equiv \chi_2>0 $,  and $\tau=0$,   it is known that  if $N\le 2$, or $N\ge 3$  and $\chi_i$ is  small relative to $b_i$ and $c_i$ ($i=1,2$), then for  any nonnegative initial data $u_0,v_0\in C^0(\bar\Omega)$, system \eqref{main-eq0} possesses a unique globally defined classical solution $(u(t,x;u_0,v_0),v(t,x;u_0,v_0), w(t,x;u_0,v_0))$ with
$u(0,x;u_0,v_0)=u_0(x)$ and $v(0,x;u_0,v_0)=v_0(x)$ (see  \cite{isra,  issh3, limuwa, st-te-wi, te-wi} and references therein).
Moreover,   the large time behaviors  of  globally defined classical solutions of \eqref{main-eq0} such as competitive exclusion, coexistence, stabilization, etc., are   investigated in  \cite{bllami, isra, issh2, issh3, miz3, miz4,  st-te-wi, te-wi}, etc.

When $\chi_1(\omega)\equiv \chi_1 >0$, $\chi_2(w)\equiv \chi_2>0 $ and $\tau=1$,
it is proved that, if
$N\le 2,$ or $N\ge 3$ and   $\chi_1$ and $\chi_2$ are small relative to other parameters in \eqref{main-eq0}, then
for any nonnegative initial data $u_0,v_0\in C(\bar\Omega)$, $w_0\in W^{1,\infty}(\Omega)$,  system \eqref{main-eq0} possesses a unique globally defined classical solution $(u(t,x;u_0,v_0),v(t,x;u_0,v_0), w(t,x;u_0,v_0))$ with
$u(0,x;u_0,v_0)=u_0(x)$, $v(0,x;u_0,v_0)=v_0(x)$, and $w(0,x;u_0,v_0)=w_0(x)$ (see \cite{ba-wi, limuwa, zhni}, etc.).
Moreover, the large-time behaviors of globally defined classical solutions are  investigated  in \cite{ba-wi, zhni}, etc.
  We refer to the readers to the articles \cite{Hor2, zhni} for the further details.

The aim of current paper is to investigate the global existence, boundedness, and combined persistence  of  classical solutions of
\eqref{main-eq0} with $\chi_1(w)=\frac{\chi_1}{w}$ and $\chi_2(w)=\frac{ \chi_2}{w} $
for some positive constants $\chi_1$ and $\chi_2$, and $\tau=0$, that is,    the following parabolic-parabolic-elliptic chemotaxis system with singular sensitivity and   Lotka-Volterra competitive kinetics,
\begin{equation}
\label{main-eq}
\begin{cases}
u_t=\Delta u-\chi_1 \nabla\cdot (\frac{u}{w} \nabla w)+u(a_1-b_1u-c_1v) ,\quad &x\in \Omega\cr
v_t=\Delta v-\chi_2 \nabla\cdot (\frac{v}{w} \nabla w)+v(a_2-b_2v-c_2u),\quad &x\in \Omega\cr
0=\Delta w-\mu w +\nu u+ \lambda v,\quad &x\in \Omega \cr
\frac{\p u}{\p n}=\frac{\p v}{\p n}=\frac{\p w}{\p n}=0,\quad &x\in\p\Omega.
\end{cases}
\end{equation}
 It  is seen that  the chemotaxis sensitivities $\frac{\chi_i}{w}$  ($i=1,2$) are  singular near $w=0$,   reflecting  an inhibition of chemotactic migration at high signal concentrations. Such a sensitivity  describing  the living organisms' response to the chemical signal  was derived by the Weber-Fechner law (see \cite{Keller-1}).

We consider classical solutions of \eqref{main-eq} with initial functions
$u_0,v_0\in C^0(\bar\Omega)$ with $u_0\ge 0$, $v_0\ge 0$, and  $\int_\Omega(u_0+v_0)>0$. Note that for such initial functions,
 if $v_0=0$ (resp. $u_0=0$), then $v(t,x)\equiv 0$
(resp. $u(t,x)\equiv 0$) on the existence interval. Note also that if $v(t,x)\equiv 0$, then \eqref{main-eq} becomes
\begin{equation}
\label{main-eq1}
\begin{cases}
u_t=\Delta u-\chi_1 \nabla\cdot (\frac{u}{w} \nabla w)+u(a_1-b_1u) ,\quad &x\in \Omega\cr
0=\Delta w-\mu w +\nu u,\quad &x\in \Omega \cr
\frac{\p u}{\p n}=\frac{\p w}{\p n}=0,\quad &x\in\p\Omega,
\end{cases}
\end{equation}
and if $u(t,x)\equiv 0$, then \eqref{main-eq} becomes
\begin{equation}
\label{main-eq2}
\begin{cases}
v_t=\Delta v-\chi_2 \nabla\cdot (\frac{v}{w} \nabla w)+v(a_2-b_2v),\quad &x\in \Omega\cr
0=\Delta w-\mu w + \lambda v,\quad &x\in \Omega \cr
\frac{\p v}{\p n}=\frac{\p w}{\p n}=0,\quad &x\in\p\Omega.
\end{cases}\,
\end{equation}
Systems \eqref{main-eq1} and \eqref{main-eq2} are essentially the same, and  are referred to as one species chemotaxis models with logistic source and singular sensitivity. They
 have  been studied in many works  (see \cite{Bil,Bla, CaWaYu, FuWiYo1,FuWiYo,HKWS,HKWS2,NaSe}, etc.). Let us briefly review  some known results for
one species chemotaxis models with logistic source and singular sensitivity.

 Consider  \eqref{main-eq1} with $a_1 = b_1 \equiv 0$ and $\mu=\nu=1$. Fujie,  Winkler, and  Yokota  in \cite{FuWiYo} proved the global existence and  boundedness of
positive classical   solutions  when $\chi_1<\frac{2}{N}$ and $N\ge 2$.  More recently, Fujie and Senba in \cite{FuSe1} proved
 the global existence and  boundedness of
classical positive  solutions  for the case of $N=2$ for any $\chi_1>0$. The existence of finite-time blow-up is then completely ruled out for any $\chi_1>0$  in the case $N=2$.  When $N\ge 3$, finite-time blow-up may occur (see \cite{NaSe}).

Consider  \eqref{main-eq1} with $a_1,b_1>0$.
Central questions include  whether the logistic source prevents  the occurrence  of    finite-time blow-up
in \eqref{main-eq1} (i.e. any positive solution exists globally);  if so,
whether  any globally defined positive solution is bounded, and
what is the long time behavior of globally defined bounded positive solutions, etc.
When $N=2$ and $a_1, b_1$ are positive constants,
it is proved in \cite{FuWiYo1} that finite-time blow-up does not occur  (see \cite[Theorem 1.1]{FuWiYo1}), and moreover,
if
\begin{equation}
\label{intro-assumption}
    a_1>
\begin{cases}
\frac{\mu \chi_1^2}{4}, &\text{if $0< \chi_1 \leq 2$}\\
\mu(\chi_1-1), &\text{if $\chi_1>2$},
\end{cases}
\end{equation}
then any globally defined positive solution is bounded.
Under some additional assumption, it is proved in \cite{CaWaYu} that  the constant solution
$(\frac{a}{b},\frac{\nu}{\mu}\frac{a}{b})$  is exponentially stable
(see \cite[Theorem 1]{CaWaYu}).
Very recently, among others,  we proved that, in any space dimensional setting,    a positive classical solution $(u(t,x),w(t,x))$ of \eqref{main-eq1} exists globally and stays bounded provided that $a_1$ is large  relative  to $\chi_1$ and the initial function { $u(0,x)$}  is not too small (see \cite[Theorem 1.2(3)]{HKWS}. Recall that, with $\chi_1\nabla \cdot(\frac{u}{w}\nabla w)$ in \eqref{main-eq1} being replaced by
$\chi_1\cdot\nabla (u\nabla w)$, then a positive classical solution $(u(t,x),w(t,x))$ exists globally provided that $b$ is large relative to $\chi_1$ { (see \cite[Theorem 2.5]{TeWi0}).}
Hence \cite[Theorem 1.2(3)]{HKWS}  reveals some  possible interesting  difference between singular and regular chemotaxis sensitivities.
 We also proved that the
  globally defined positive solutions  of \eqref{main-eq1} are away from $0$ provided that $a_1$ is large  relative  to $\chi_1$ and the initial functions are not too small
(see \cite[Theorem 1.3]{HKWS2}).

However, as far as we know,  there is little study on  the two-species
chemotaxis system \eqref{main-eq}.
It is the aim of this paper  to investigate the global existence, boundedness, and combined persistence  of  classical solutions of
\eqref{main-eq}.

\begin{definition}
\label{solu-def}For given  $u_0(\cdot)\in C^0(\bar\Omega)$ and $v_0(\cdot)\in C^0(\bar\Omega)$ satisfying that
$u_0\ge 0$, $v_0\ge 0$, and $\int_\Omega (u_0(x)+v_0(x))dx>0$,
 we say $(u(t,x),v(t,x),w(t,x))$ is a {\rm positive classical solution} of \eqref{main-eq} on $(0,T)$ for some $T\in (0,\infty]$ with initial condition  $(u(0,x),v(0,x))=(u_0(x),v_0(x))$ if
 \begin{equation}
\label{positive-solu-eq}
{u(t,x;u_0,v_0)+v(t,x;u_0,v_0)>0,\,\, w(t,x;u_0,v_0)>0\quad \forall t\in (0,T),\,\, x\in\bar\Omega,}
\end{equation}
\begin{equation*}
u(\cdot,\cdot),v(\cdot,\cdot)\in  C([0,T)\times\bar\Omega )\cap C^{1,2}(  (0,T)\times\bar\Omega),\quad
w(\cdot,\cdot)\in C^{0,2}((0,T)\times \bar\Omega),
\end{equation*}
\begin{equation*}
\lim_{t\to 0+}\|u(t,\cdot)-u_0(\cdot)\|_{C^0(\bar\Omega)}=0,\quad \lim_{t\to 0+}\|v(t,\cdot)-v_0(\cdot)\|_{C^0(\bar\Omega)}=0,
\end{equation*}
and $(u(t,x),v(t,x),w(t,x))$ satisfies \eqref{main-eq} for all $(t,x)\in (0,T)\times \Omega$.
\end{definition}

 The following proposition on the local existence of classical solutions of \eqref{main-eq} can be proved by the similar arguments as those in \cite[Lemma 2.2]{FuWiYo1}.

\begin{proposition} [Local existence]
\label{local-existence-prop}
{For given  $u_0(\cdot)\in C^0(\bar\Omega)$ and $v_0(\cdot)\in C^0(\bar\Omega)$ satisfying that
$u_0\ge 0$, $v_0\ge 0$, and $\int_\Omega (u_0(x)+v_0(x))dx>0$,  there exists $T_{\max}(u_0,v_0)\in (0,\infty]$
such that  \eqref{main-eq} has a unique positive  classical solution, denoted by $(u(t,x;u_0,v_0)$, $v(t,x;u_0,v_0)$, $w(t,x;u_0,v_0))$,  on $(0,T_{\max}(u_0,v_0))$ with initial condition $(u(0,x;u_0,v_0),v(0,x;u_0,v_0))=(u_0(x),v_0(x))$.  { Moreover, if $\int_\Omega u_0(x)dx>0$ and $\int_\Omega v_0(x)dx>0$, then
$$
u(t,x;u_0,v_0)>0\quad {\rm and}\quad v(t,x;u_0,v_0)>0\quad \forall\, t\in (0,T_{\max}(u_0,v_0)),\,\, x\in\bar\Omega.
$$ }
If $T_{\max}(u_0,v_0)< \infty,$ then either }
\begin{equation*}
\limsup_{t \nearrow T_{\max}(u_0,v_0)} \left(\left\| u(t,\cdot;u_0,v_0)+v(t,\cdot;u_0,v_0) \right\|_{C^0(\bar \Omega)} \right) =\infty,
\end{equation*}
or
\begin{equation}
\label{local-infty-2}
    \liminf_{t \nearrow T_{\max}(u_0,v_0)} \inf_{x \in \Omega} w(t,x;u_0,v_0) =0.
\end{equation}
\end{proposition}
We will focus on the following problems in this paper: whether  $(u(t,x;u_0,v_0),v(t,x;u_0,v_0)$, $w(t,x;u_0,v_0))$  exists globally, i.e., $T_{\max}(u_0,v_0)=\infty$, for any $u_0(\cdot)\in C^0(\bar\Omega)$ and $v_0(\cdot)\in C^0(\bar\Omega)$ satisfying that
$u_0\ge 0$, $v_0\ge 0$, and
$\int_\Omega (u_0(x)+v_0(x))dx>0$; If $T_{\max}(u_0,v_0)=\infty$, whether  $(u(t,x;u_0,v_0),v(t,x;u_0,v_0)$, $w(t,x;u_0,v_0))$ is bounded above and stays away from $0$.

We point out that in the study of global existence and boundedness of classical solutions of  one species chemotaxis model \eqref{main-eq1} with singular sensitivity, it is crucial to prove the boundedness  of $(\int_\Omega u(t,x)dx)^{-1}$ and  the boundedness  of $\int_\Omega u^p(t,x)dx$ for some $p\gg  1$.
The novel idea discovered in this paper for the study of global existence and boundedness of classical solutions of
 \eqref{main-eq} is to prove the boundedness of $(\int_\Omega(u(t,x)+v(t,x))dx)^{-1}$ and the boundedness of
$\int_\Omega(u(t,x)+v(t,x))^p dx$ for $p\gg 1$.  Note that for \eqref{main-eq},
$(\int_\Omega u(t,x)dx)^{-1}$ (resp. $(\int_\Omega v(t,x)dx)^{-1}$) may not be bounded.
Note also that the boundedness of $(\int_\Omega(u(t,x)+v(t,x))dx)^{-1}$ is strongly related to the boundedness of $\int_\Omega(u(t,x)+v(t,x))^{-q}dx)$ for some $q>0$.

In the rest of the introduction, we introduce some standing notations { and assumptions} in subsection 1.1,  and state the main results of the paper and provide some remarks on the main results in subsection 1.2.

\subsection{Notations and assumptions}

In this subsection, we introduce some standing notations { and assumptions} to be used throughout the paper.

Observe that for given  $u_0(\cdot)\in C^0(\bar\Omega)$ and $v_0(\cdot)\in C^0(\bar\Omega)$ satisfying that
$u_0\ge 0$, $v_0\ge 0$, and $\int_\Omega (u_0(x)+v_0(x))dx>0$, if $v_0\equiv 0$ (resp. $u_0\equiv 0$), then
$v(t,x;u_0,v_0)\equiv 0$  (resp. $u(t,x;u_0,v_0)\equiv 0$)  for $t\in (0,T_{\max}(u_0,v_0))$
and $(u(t,x;u_0),w(t,x;u_0)):=(u(t,x;u_0,0),w(t,x;u_0,0))$ (resp.
$v(t,x;v_0)$, $w(t,x;v_0)):=(v(t,x;0,v_0)$, $w(t,x;0,v_0))$)  is the  solution of \eqref{main-eq1} (resp. \eqref{main-eq2})
with initial condition $u(0,x;u_0)=u_0(x)$ (resp. $v(0,x;v_0)=v_0(x)$). Hence,
throughout the rest of this paper, we consider classical solutions of \eqref{main-eq} with the initial function $u_0(x), v_0(x)$  satisfying
\begin{equation}
\label{initial-cond-eq}
u_0,v_0 \in C^0(\bar{\Omega}), \quad u_0, v_0 \ge 0, \quad  {\rm and} \quad  \int_\Omega u_0 >0,\,\,\, \int_\Omega v_0>0.
\end{equation}

Let
\begin{equation*}
\begin{cases}
a_{\min}=\min\{a_1,a_2\},\quad & a_{\max}=\max\{a_1,a_2\}\cr
b_{\min}=\min\{b_1,b_2\}, \quad  & b_{\max}=\max\{b_1,b_2\}\cr
c_{\min}=\min\{c_1,c_2\},\quad & c_{\max}=\max\{c_1,c_2\}.
\end{cases}
\end{equation*}
For given $B>0$ and $\beta\not = { \chi_2-B}$, let
\begin{equation}
\label{function-f-eq}
f(\mu,\chi_1,\chi_2, \beta, B)=\mu(B+\beta) \Big(1+\frac{{B}({ \chi_2-B} -\beta)^2+(\chi_1-\chi_2)^2\beta}{4B\beta}\Big).
\end{equation}
For fixed $\mu>0$, $\chi_1>0$ and $\chi_2>0$, let
\begin{align*}
\chi_1^*(\mu,\chi_1,\chi_2)=\inf\Big\{f(\mu,\chi_1,\chi_2,\beta,B)\,|\, &B>0, \beta>0,\beta\not ={ \chi_2-B}\Big\}.
\end{align*}
Similarly, let
\begin{align*}
\chi_2^*(\mu,\chi_1,\chi_2)=\inf\Big\{f(\mu,\chi_2,\chi_1,\beta,B)\,|\, &B>0, \beta>0,\beta\not ={ \chi_1-B}\Big\}.
\end{align*}
Let
\begin{equation}
\label{bound-for-a-eq3}
{ \chi^*(\mu,\chi_1,\chi_2)}=\min\{\chi_1^*(\mu,\chi_1,\chi_2),\chi_2^*(\mu,\chi_1,\chi_2)\},
\end{equation}
{and
\begin{equation*}
\chi^{**}(\mu,\chi_1,\chi_2)=\begin{cases}
 { 2\mu\,  \max\{\chi_1,\chi_2\} }\quad & {\rm if}\; (\chi_1-\chi_2)^2\le { \max\{ 4\chi_1,4\chi_2\} }\cr
{ \mu \, {  \max}}\{\chi_1,\chi_2\}+\frac{\mu (\chi_1-\chi_2)^2}{4} \quad &{\rm if}\; (\chi_1-\chi_2)^2>{ \max\{4\chi_1,4\chi_2\}}.
\end{cases}
\end{equation*} }
Note that
$$
\chi^*(\mu,\chi_1,\chi_2)\le \chi^{**}(\mu,\chi_1,\chi_2)
$$
(see \eqref{main-assumption-1}).
The numbers $\chi^*(\mu,\chi_1,\chi_2)$ and {$\chi^{**}(\mu,\chi_1,\chi_2)$} will be used to get lower bounds for $\int (u(t,x;u_0,v_0)+v(t,x;u_0,v_0)$, which are essential in the proofs of global existence of $(u(t,x;u_0,v_0)$, $v(t,x;u_0,v_0),w(t,x;u_0,v_0))$.
{Let
\begin{equation*}
q^*=q_{\chi_1,\chi_2}=\begin{cases}
1\quad &{\rm if}\; (\chi_1-\chi_2)^2\le {\max} \{4\chi_1,4\chi_2\},\cr\cr
{  \max}\Big\{\frac{4\chi_1}{(\chi_1-\chi_2)^2},\frac{4\chi_2}{(\chi_1-\chi_2)^2}\Big\}\quad &{\rm if}\; (\chi_1-\chi_2)^2>{ \max}\{4\chi_1,4\chi_2\},
\end{cases}
\end{equation*}
and
\begin{equation*}
\chi_{1,2}= \min\Big\{\chi_1^2+(\chi_1-\chi_2)^2,\chi_2^2+(\chi_1-\chi_2)^2\Big\},
\end{equation*}
{{ as well as
\begin{equation*}
C_{\chi_1,\chi_2}=\begin{cases}
1\quad &{\rm if}\; \chi_{1,2}\ge 1\cr
\sqrt{\chi_{1,2}}\quad &{\rm if}\; \chi_{1,2}<1.
\end{cases}
\end{equation*}}

{Note that there is $\delta_0>0$ such that for any $(u_0,v_0)$ satisfying \eqref{initial-cond-eq},
\begin{equation}
\label{w-lower-bound-eq1}
    w (t,x;u_0,v_0)\ge \delta_0  \int_{\Omega} (u(t,x;u_0,v_0) + v(t,x;u_0,v_0)) >0 \quad \forall\, t\in [0,T_{\max}(u_0,v_0)),\,\, x\in\Omega
\end{equation}
(see  Lemma  \ref{pre-lm-1}).
Note also that, for any  $p \ge  3$ and $p-\sqrt{2p-3} < k <  p+\sqrt{2p-3}$,  there are $M(p,k)>0$ and $\tilde M(p,k)>0$ such that
 \begin{align}
\label{w-upper-bound-eq1}
   \int_{\Omega} \frac{|\nabla w(t,x;u_0,v_0)|^{2p}}{w^{k}(t,x;u_0,v_0)} \leq & M(p,k)\int_{\Omega} \frac{(\nu u(t,x;u_0,v_0)+\lambda v(t,x;u_0,v_0))^p}{w^{k-p}(t,x;u_0,v_0)} \nonumber \\
&+ \tilde M(p,k) \int_{\Omega} w^{2p-k}(t,x;u_0,v_0)\quad \forall\,\, t \in (0, T_{\rm max}(u_0,v_0))
\end{align}
(see Proposition \ref{new-main-prop2}).
}

At some places, we impose the following conditions  on the parameters in  \eqref{main-eq}
and  on $(u_0,v_0)$.

\medskip

\noindent {\bf (H1)} \, {
${\displaystyle   a_{\min}>\chi^{**}(\mu,\chi_1,\chi_2)+\frac{{\big(\chi_{1,2} \cdot \max\{\nu,\lambda\} \cdot{ N}\big)^{q^*}} \big(M(N+1,N+1)\big)^{\frac{q^*}{N+1}}  \big(b_{\max}+c_{\max}\big)  a_{\max}^{1-q^*} }{(2 \delta_0|\Omega| { C_{\chi_1,\chi_2}})^{q^*}\cdot  \min\{b_{\min},c_{\min}\}}.}
$}

\medskip

\noindent {\bf (H2)} $  (u_0,v_0)$ {satisfies}  \eqref{initial-cond-eq} and  there is $\tau_0\in [0,T_{\max}(u_0,v_0))$ such that
\begin{equation}
\label{new-initial-cond-eq1}
\int_\Omega \big(u+v\big)^{-1}(\tau_0,x;u_0,v_0)\le  \frac{(b_{\max}+c_{\max})|\Omega|}{\big(a_{\min}-\chi^{**}(\mu,\chi_1,\chi_2)\big)\cdot {C_{\chi_1,\chi_2}}}
\end{equation}
in the case that  $ (\chi_1-\chi_2)^2\le { \max\{4\chi_1,4\chi_2\}}$, and
\begin{equation}
\label{new-initial-cond-eq2}
\begin{cases}
 \int_\Omega\big (u+v\big)(\tau_0,x;u_0,v_0)\le   \frac{a_{\max}|\Omega|}{\min\{b_{\min},c_{\min}\} \cdot { C_{\chi_1,\chi_2}^{\frac{q^*}{1-q^*}}}},     \cr\cr
 \int_\Omega\big(u+v\big)^{-q^*}(\tau_0,x;u_0,v_0) \le  \frac{(b_{\max}+c_{\max})|\Omega|\big(a_{\max}|\big)^{1-q^*}}{\big(a_{\min}-\chi^{**}(\mu,\chi_1,\chi_2)\big)\cdot \big(\min\{b_{\min},c_{\min}\}\big)^{1-q^*} \cdot { C_{\chi_1,\chi_2}^{q^*}}}
\end{cases}
\end{equation}
 in the case that  $ (\chi_1-\chi_2)^2>{ \max}\{{ 4}\chi_1,4\chi_2\}$.

\smallskip

Note that  the assumption  {\bf (H1)} indicates that $a_1,a_2$ are large relative to $\chi_1,\chi_2$,   which  would prevent $w$ becomes too small as time evolutes and is a biologically {meaningful} condition. { We point out that, when $\chi_1,\chi_2\to 0$, {\bf (H1)} becomes
$$
a_{\min}>0,
$$
which always holds.}
 The assumption  {\bf (H2)} includes implicitly the condition $a_{\min}>\chi^{**}(\mu,\chi_1,\chi_2)$.  The condition \eqref{new-initial-cond-eq1}    indicates that  $u_0+v_0$ is not  small   in the case that $\chi_1$ and $\chi_2$ are close in the sense  that
$(\chi_1-\chi_2)^2\le  { \max}\{4\chi_1,4\chi_2\}$.
 The condition   \eqref{new-initial-cond-eq2} indicates that  $u_0+v_0$ is neither  big nor  small   in the case that $\chi_1$ and $\chi_2$ are not close in the sense that $(\chi_1-\chi_2)^2>{ \max}\{4\chi_1,4\chi_2\}$.
Both  \eqref{new-initial-cond-eq1} and  \eqref{new-initial-cond-eq2}  would  also prevent $w$ becomes too small as time evolutes and are biologically meaningful conditions.
We point out that when $(\chi_1-\chi_2)^2\le { \max}\{4\chi_1,4\chi_2\}$, it is not required that $u_0+v_0$ is not big.
{We also point out that, when $0<\chi_1<1$ and $0<\chi_2<1$,  we always have
$$
(\chi_1-\chi_2)^2\le \max\{4\chi_1,\chi_2\}.
$$
Hence, when $0<\chi_1<1$ and $0<\chi_2<1$, {\bf (H2)} becomes
\eqref{new-initial-cond-eq1}, and when $\chi_1,\chi_2\to 0$, \eqref{new-initial-cond-eq1}
becomes
$${ \int_\Omega (u+v)^{-1}(\tau_0,u_0,v_0)<\infty}
$$
for some $\tau_0>0$, which always holds.}
}

\subsection{Main results and remarks}

In this subsection, we state the main results of the paper and provide some remarks on the main results.

Observe that, for any given $u_0,v_0$ satisfying \eqref{initial-cond-eq},
by the third equation in \eqref{main-eq},
$$
\mu\int_\Omega w(t,x;u_0,v_0)dx=\int_\Omega (\nu u(t,x;u_0,v_0)+\lambda v(t,x;u_0,v_0))dx
$$
and
\begin{align*}
\min\{\nu,  \lambda\}\int_\Omega (u(t,x;u_0,v_0)+v(t,x;u_0,v_0))dx&\le \int_\Omega (\nu u(t,x;u_0,v_0)+\lambda v(t,x;u_0,v_0))dx\\
&\le\max\{\nu,\lambda\}\int_\Omega (u(t,x;u_0,v_0)+v(t,x;u_0,v_0))dx
\end{align*}
for any $t\in (0,T_{\max}(u_0,v_0))$. Hence $\liminf_{t\nearrow T_{\max}(u_0,v_0)}\int_\Omega (u(t,x;u_0,v_0)+v(t,x;u_0,v_0))dx=0$ implies that \eqref{local-infty-2} holds. On the other hand,  by \eqref{w-lower-bound-eq1},
$$
    \liminf_{t \nearrow T_{\max}(u_0,v_0)} \inf_{x \in \Omega} w(t,x;u_0,v_0) =0\,\,  {\rm iff}
\,\,  \liminf_{t\nearrow T_{\max}(u_0,v_0)}\int_\Omega (u(t,x;u_0,v_0)+v(t,x;u_0,v_0))dx=0.
$$

The first main theorem of the current  paper  is on the lower bounds of the combined mass $\int_\Omega(u(t,x;u_0,v_0)+v(t,x;u_0,v_0))dx$
on any bounded subinterval of $(0,T_{\max}(u_0,v_0))$, which would provides the upper bounds of $\big(\int_\Omega(u(t,x;u_0,v_0)+v(t,x;u_0,v_0))dx\big)^{-1}$ on bounded subintervals of $(0,T_{\max}(u_0,v_0))$.

\begin{theorem}[Local lower bound of the combined mass]
\label{new-main-thm1}
For any $T\in (0,\infty)$ and  any $u_0,v_0$ satisfying \eqref{initial-cond-eq},
\begin{equation}
\label{mass-persistence-eq1}
\inf_{0\le t<\min\{T, T_{\max}(u_0,v_0)\}}\int_\Omega \big(u(t,x;u_0,v_0)+v(t,x;u_0,v_0)\big)dx>0.
\end{equation}
\end{theorem}

\begin{remark}
\label{rk-1} {\rm
\begin{itemize}
\item[(1)]
By Proposition \ref{local-existence-prop},   \eqref{w-lower-bound-eq1} and \eqref{mass-persistence-eq1}, if $T_{\max}(u_0,v_0)<\infty$, then  we must have
\begin{equation*}
\limsup_{t \nearrow T_{\max}(u_0,v_0)} \left(\left\| u(t,\cdot;u_0,v_0)+v(t,\cdot;u_0,v_0) \right\|_{C^0(\bar \Omega)} \right) =\infty.
\end{equation*}

\item[(2)] If
\begin{equation*}
\inf_{0\le t<\min\{T, T_{\max}(u_0,v_0)\}}\int_\Omega u(t,x;u_0,v_0)dx>0
\end{equation*}
 and
\begin{equation*}
\inf_{0\le t<\min\{T, T_{\max}(u_0,v_0)\}}\int_\Omega v(t,x;u_0,v_0)dx>0,
\end{equation*}
then  \eqref{mass-persistence-eq1} holds. But the converse may not be true because  competitive exclusion may occur in \eqref{main-eq}, which will be studied somewhere else.

\item[(3)] {It will be proved that $T_{\max}(u_0,v_0)=\infty$  provided that    {\bf (H1)}  and {\bf (H2)} hold (see Theorem \ref{new-main-thm4}).}
\end{itemize}}
\end{remark}

The second  main theorem  of the current  paper is on the local  {$L^p$- and}  $C^\theta$-boundedness of positive classical solutions of \eqref{main-eq}.

\begin{theorem} [Local  { $L^p$- and} $C^\theta$-boundedness]
\label{new-main-thm2}
$\quad$

\begin{itemize}

\item[(1)]   (Local $L^p$-boundedness)
 Assume that  {\bf (H1)} holds. There are $p^*>\max\{2,N\}$, $\tilde p^*>1$,  and  $M_1> 0$  such that  for any $(u_0,v_0)$ satisfies {\bf (H2)},  there holds
\vspace{-0.05in} \begin{align*}
\int_\Omega (u+v)^{p}&(t,x;u_0,v_0)dx\le  e^{-(t-{ \tau_0}) }\int_\Omega (u+v)^{p}(\tau_0,x;u_0,v_0)dx \\
& +M_1 e^{-(t-\tau_0)}(t-\tau_0)\Big(\int_\Omega(u+v)^{p^*}(\tau_0,x;u_0,v_0)\Big)^{\tilde p^*}+M_1,\,\, \forall\, \tau_0<t<T_{\max}(u_0,v_0),
\end{align*}
where $p=p^*$ or $2p^*$.

\item[(2)] (Local $C^\theta$-boundedness) For  any $p>2N$ and  $0<\theta<1-\frac{2N}{p}$,  there are $M_2>0$, $\beta>0$, and $\gamma>0$ such that for
 any   { $(u_0,v_0)$ satisfies \eqref{initial-cond-eq} and $T\in (0,\infty)$},  there holds
\begin{align}
\label{new-infinity-bdd-eq0}
&\|u(t,\cdot;u_0,v_0)+v(t,\cdot;u_0,v_0)\|_{C^\theta(\bar\Omega)}\nonumber\\
&\le  M_2 \Big[ (t-\tau)^{-\beta} e^{-\gamma (t-\tau)} \|u(\tau,\cdot;u_0,v_0)+v(\tau,\cdot;u_0,v_0)\|_{L^p}\nonumber\\
&\,\,\,\,+\frac{\displaystyle \sup_{\tau \le t <\hat T}\|u(t,\cdot;u_0,v_0)+
v(t,\cdot;u_0,v_0)\|_{L^p}^2}{\displaystyle \inf_{\tau \le t< \hat T, x\in\Omega} w(t,x;u_0,v_0)}\nonumber\\
&\,\,  \,\,+\sup_{\tau \le t<\hat T}\|u(t,\cdot;u_0,v_0)+v(t,\cdot;u_0,v_0)\|_{L^p}+1\Big]
\end{align}
for any $0<\tau<t<\hat T=\min\{T,T_{\max}(u_0,v_0)\}$.
\end{itemize}
\end{theorem}

\begin{remark}
\label{rk-2}
$\quad$
{\rm
\begin{itemize}
\item[(1)]
 By Theorem \ref{new-main-thm2}{(2)}, if $\inf_{x\in \Omega} w(t,x;u_0,v_0)$ is bounded away from zero on $(0, T_{\max}(u_0,v_0))$ {and  $\int_\Omega(u+v)^p(t,x;u_0,v_0)dx$ is bounded on $(0,T_{\max}(u_0,v_0))$ for some $p>2N$} ,
then $\|u(t,\cdot;u_0,v_0)+v(t,\cdot;u_0,v_0)\|_{C^\theta(\bar\Omega)}$ is bounded on $[\tau,T_{\max}(u_0,v_0))$ for some $0<\theta<1$ and any $0<\tau<T_{\max}(u_0,v_0)$.

\item[(2)]
By \eqref{w-lower-bound-eq1}, $w(t,x;u_0,v_0)$ is bounded below by $\int_\Omega (u(t,x;u_0,v_0)+v(t,x;u_0,v_0))dx$.
By Theorems \ref{new-main-thm1} and \ref{new-main-thm2}{(2)},  $\|u(t,\cdot;u_0,v_0)+v(t,\cdot;u_0,v_0)\|_{C^\theta (\bar\Omega)}$ is bounded on $[\tau,\hat T)$ for some $\theta\in (0,1)$,  any $T\in (0,\infty)$,  and  any $0<\tau<\hat T:=\min\{T,T_{\max}(u_0,v_0)\}$,
which plays an important role in the study of global existence of classical solutions of \eqref{main-eq}.
\end{itemize}}
\end{remark}

The third main theorem  of the current paper  is on  global existence of classical solutions of \eqref{main-eq}.

\begin{theorem} [Global existence]
\label{new-main-thm3}
{ Assume that {\bf (H1)} holds.
For   any $u_0,v_0$ satisfying {\bf (H2)}, the solution  $(u(t,x;u_0,v_0),v(t,x;u_0,v_0),w(t,x;u_0,v_0))$ exists globally, that is, $$T_{\max}(u_0,v_0)=\infty.$$}
\end{theorem}

\begin{remark}
\label{rk-3}  {\rm The proof of Theorem \ref{new-main-thm3} relies on Theorems \ref{new-main-thm1} and \ref{new-main-thm2}.  The proof of Theorem \ref{new-main-thm2}(1) replies on estimates on the lower bounds of $w(t,x;u_0,v_0)$, which relies on estimates on the upper bounds of $\int_\Omega(u+v)^{-q}$ for some $q>0$.}
\end{remark}

The fourth main theorem of the current paper is on the  boundedness
of $\int_\Omega (u+v)^{-q}$  for some $q>0$.  For given $u_0,v_0$ satisfying \eqref{initial-cond-eq} and $\tau\in [0,T_{\max}(u_0,v_0))$, let
\begin{equation*}
 m^*(\tau,u_0,v_0):= {\rm max}\Big\{\int_{\Omega}(u(\tau,x;u_0,v_0)+ v(\tau,x;u_0,v_0))dx,  \frac{a_{\max}|\Omega|}{\min\{b_{\min},c_{\min}\}} \Big\} .
\end{equation*}

\begin{theorem}
\label{new-main-thm4}
\begin{itemize}
\item[(1)]  (Boundedness of $\int_\Omega (u+v)^{-q}$) $\,\,$
 Assume that $a_{\min}>\chi^*(\mu,\chi_1,\chi_2)$. Then there  is $q>0$    such that  for any $u_0,v_0$ satisfying \eqref{initial-cond-eq} and for every  $0<\tau<T_{\max}(u_0,v_0)$,
\begin{align*}
&\int_\Omega (u(t,x;u_0,v_0)+v(t,x;u_0,v_0))^{-q}dx\nonumber\\
& \le e^{-\frac{\epsilon_0 q}{2}(t-\tau)} \int_\Omega(u(\tau,x;u_0,v_0)+v(\tau,x;u_0,v_0))^{-q}dx+  2 C_{\epsilon_0,\tau,u_0,v_0} \epsilon_0^{-1}
\end{align*}
for all $\tau<t<T_{\max}(u_0,v_0)$, where $\epsilon_0=a_{\min}-\chi^*(\mu,\chi_1,\chi_2)$ and
\begin{align*}
    C_{\epsilon_0, \tau,u_0,v_0}=\begin{cases}
(b_{\max}+c_{\max})|\Omega| &{\rm if}\,\,\,q=1,\cr
      \frac{(b_{\max}+c_{\max})^q}{q}\big(\frac{4}{\epsilon_0}\big)^{q}\big(\frac{q-1}{q}\big)^{q-1}|\Omega| & {\rm if} \;\,\,q> 1, \cr
   (b_{\max}+c_{\max})|\Omega|^q (m^*(\tau,u_0,v_0))^{1-q} \; & {\rm if} \; \,\, q < 1.
    \end{cases}
\end{align*}

\item[(2)] (Boundedness of $\int_\Omega (u+v)^{-q}$) $\,\,$
Assume that $a_{\min}>\chi^{**}(\mu,\chi_1,\chi_2)$. Then
for any $u_0,v_0$ satisfying \eqref{initial-cond-eq} and for every  $0<\tau<T_{\max}(u_0,v_0)$,   the following holds,
\begin{align*}
&\int_\Omega (u(x,t;u_0,v_0)+v(x,t;u_0,v_0))^{-{ q}}dx\nonumber\\
& \le e^{-(a_{\min}-\chi^{**}(\mu,\chi_1,\chi_2)) { q }(t-\tau)} \int_\Omega(u(x,\tau;u_0,v_0)+v(x,\tau;u_0,v_0))^{-{ q}}dx+ \tilde   C_{\tau,u_0,v_0}
\end{align*}
for all  $\tau<t<T_{\max}(u_0,v_0)$,
where {$q=q_{\chi_1,\chi_2}$ and}
\begin{equation*}
\tilde C_{\tau,u_0,v_0}=
\begin{cases}
\frac{(b_{\max}+c_{\max})|\Omega|}{a_{\min}-\chi^{**}(\mu,\chi_1,\chi_2)}\,\, &{\rm if}\; q_{\chi_1,\chi_2}=1,\cr\cr
\frac{(b_{\max}+c_{\max})|\Omega|^{q_{\chi_1,\chi_2}}\big(m^*(\tau,u_0,v_0)\big)^{1-q_{\chi_1,\chi_2}}}{a_{\min}-\chi^{**}(\mu,\chi_1,\chi_2)}  \,\,  &{\rm if}\; q_{\chi_1,\chi_2}<1.
\end{cases}
\end{equation*}

\end{itemize}
\end{theorem}

\begin{remark}
\label{rk-3*}
\begin{itemize}
\item[(1)]  The condition $a_{\min}>\chi^{**}(\mu,\chi_1,\chi_2)$ in Theorem \ref{new-main-thm4}(2) implies the condition $a_{\min}>\chi^*(\mu,\chi_1,\chi_2)$ in Theorem { \ref{new-main-thm4}(1)}, but the $q$ in  Theorem {\ref{new-main-thm4}(2)} is explicit.

\item[(2)]  Assume that $a_{\min}>\chi^{*}(\mu,\chi_1,\chi_2)$  or $\chi^{**}(\mu,\chi_1,\chi_2)$.  By H\"older's inequality,   for any $q>0$,
\begin{equation}
\label{lower-bound-u+v-eqq0}
\int_\Omega(u(t,x;u_0,v_0)+v(t,x;u_0,v_0))\ge \frac{|\Omega|^{\frac{q+1}{q}}}{\Big(\int_\Omega(u(t,x;u_0,v_0)+v(t,x;u_0,v_0))^{-q}\Big)^{\frac{1}{q}}}
\end{equation}
for all  $0< t<T_{\max}(u_0,v_0)$.
By   \eqref{lower-bound-u+v-eqq0}, and Theorem \ref{new-main-thm4}(1) or Theorem \ref{new-main-thm4}(2),
$$
\inf_{0\le t<T_{\max}(u_0,v_0)}\int_\Omega (u(t,x;u_0,v_0)+v(t,x;u_0,v_0))dx>0,
$$
which improves \eqref{mass-persistence-eq1}.

\item[(3)] Assume  that {\bf (H1)} and {\bf (H2)} hold. By Theorem  \ref{new-main-thm4},
$\int_\Omega(u(t,x;u_0,v_0)+v(t,x;u_0,v_0))^{-q}dx$  stays bounded for some $q>0$.
{This  plays important roles  in the study of $L^p$-boundedness and  global existence of   classical solutions  of \eqref{main-eq}. We will therefore prove Theorem \ref{new-main-thm4} before proving Theorems \ref{new-main-thm2} and \ref{new-main-thm3}.}
\end{itemize}
\end{remark}

The last main theorem of the current paper  is about  initial independent ultimate upper and lower bounds   of classical solutions of \eqref{main-eq}, some of which follows from Theorems \ref{new-main-thm2}-\ref{new-main-thm4} quite directly. We state them in a theorem since
they have important impacts on the understanding of the asymptotic behavior of  globally defined positive solutions of \eqref{main-eq}.

\begin{theorem}
\label{new-main-thm5}
Assume { that {\bf (H1)} holds}. Then the following hold.

\begin{itemize}

\item[(1)]  (Uniform boundedness)
Let  $q>0$ be as in Theorem \ref{new-main-thm4}(2) and ${ p=p^*>\max\{2,N\}}$, where $p^*$ is  as in Theorem \ref{new-main-thm2}(1). Let ${0<\theta< 1-\frac{N}{p}}$.  There are $M_1^*>0$,  $M_2^*>0$ and $M_3^*>0$ such that for any $u_0,v_0$ satisfying { {\bf (H2)}},
\begin{equation}
\label{new-uniform-bd-eq1}
\limsup_{t\to\infty} \int_\Omega (u(t,x;u_0,v_0)+v(t,x;u_0,v_0))^{-q}dx\le M_1^*,
\end{equation}
\begin{equation}
\label{new-uniform-bd-eq2}
\limsup_{t\to\infty} \int_\Omega (u(t,x;u_0,v_0)+v(t,x;u_0,v_0))^p dx\le M_2^*,
\end{equation}
and
\begin{equation}
\label{new-uniform-bd-eq3}
\limsup_{t\to\infty} \|u(t,\cdot;u_0,v_0)+v(t,\cdot;u_0,v_0)\|_{C^\theta(\bar\Omega)}\le M_3^*.
\end{equation}

\item[(2)]  (Combined pointwise persistence) There is $M_0^*>0$ such that for any $u_0,v_0$ satisfying { {\bf (H2)}},
\begin{equation}
\label{new-uniform-bd-eq4}
 \liminf_{t\to\infty}\inf_{x\in\Omega}(u(t,x;u_0,v_0)+v(t,x;u_0,v_0))\ge M_0^*.
\end{equation}

\end{itemize}
\end{theorem}

\begin{remark}
\label{rk-5}
\begin{itemize}
\item[(1)]  Let  $p, q$ be as in Theorem \ref{new-main-thm5}. Let
\begin{align*}
\mathcal{E}=\Big\{u,v\in C^0(\bar\Omega)\,|\, & u\ge 0, \,\, v\ge 0,\,\, \int_\Omega (u(x)+ v(x))dx >0,\nonumber\\
&\int_\Omega (u(x)+v(x))^{-q}\le M_1^*,\,\, \int_\Omega (u(x)+v(x))^p dx\le M_2^*\Big\}.
\end{align*}
{ If {\bf (H1)} holds,} Theorem \ref{new-main-thm5} shows that $\mathcal{E}$ eventually attracts all globally defined positive solutions of \eqref{main-eq} { with initial conditions satisfying {\bf (H2)}}.

\item[(2)] \eqref{new-uniform-bd-eq4} implies that,  if { {\bf (H1)} holds}, then there is $m_0^*>0$ such that for any $u_0,v_0$ satisfying {{\bf (H2)}},
\begin{equation}
\label{aux-new-eq00-4}
\liminf_{t\to\infty}\int_\Omega (u(t,x;u_0,v_0)+v(t,x;u_0,v_0))dx\ge m_0^*,
\end{equation}
which is refereed to as combined mass persistence.
It  remains open whether $T_{\max}(u_0,v_0)=\infty$ and  \eqref{aux-new-eq00-4}  holds without the assumption  {
{\bf (H2)}}.

\item[(3)]  \eqref{new-uniform-bd-eq3} implies that,   { if {\bf (H1)} holds}, then  for any $u_0,v_0$ satisfying { {\bf (H2)}},
\begin{equation}
\label{aux-new-eq00-5}
\limsup_{t\to\infty}\|u(t,\cdot;u_0,v_0)+v(t,\cdot;u_0,v_0)\|_{L^\infty(\Omega)}\le M_3^*.
\end{equation}
It also remains open whether  $T_{\max}(u_0,v_0)=\infty$ and  \eqref{aux-new-eq00-5}  holds without  {  the assumption
{\bf (H2)}.}

\item[(4)]  We say {\rm both species persistent in mass} if for any $u_0,v_0$ satisfying \eqref{initial-cond-eq},
$$
\liminf_{t\to\infty} \int_\Omega u(t,x;u_0,v_0)dx>0\,\, {\rm and}\,\, \liminf_{t\to\infty} \int_\Omega v(t,x;u_0,v_0)dx>0,
$$
and say {\rm competitive exclusion in mass} occurs if for any $u_0,v_0$ satisfying \eqref{initial-cond-eq},
$$
\limsup_{t\to\infty} \int_\Omega u(t,x;u_0,v_0)dx=0\,\, {\rm and}\,\,  \liminf_{t\to\infty}\int_\Omega  v(t,x;u_0,v_0)dx>0
$$
or
$$
\liminf_{t\to\infty} \int_\Omega u(t,x;u_0,v_0)dx>0\,\, {\rm and}\,\,  \limsup_{t\to\infty} \int_\Omega v(t,x;u_0,v_0)dx=0.
$$
 We say {\rm both species persistent pointwise} if for any $u_0,v_0$ satisfying \eqref{initial-cond-eq},
$$
\liminf_{t\to\infty} \inf_{x\in \Omega} u(t,x;u_0,v_0)>0\,\, {\rm and}\,\, \liminf_{t\to\infty} \inf_{x\in \Omega} v(t,x;u_0,v_0)>0,
$$
and say {\rm pointwise competitive exclusion} occurs  if for any $u_0,v_0$ satisfying \eqref{initial-cond-eq},
$$
\limsup_{t\to\infty} \sup_{x\in \Omega} u(t,x;u_0,v_0)=0\,\, {\rm and}\,\,  \liminf_{t\to\infty}\inf_{x\in\Omega}  v(t,x;u_0,v_0)>0
$$
or
$$
\liminf_{t\to\infty} \inf_{x\in\Omega} u(t,x;u_0,v_0)>0\,\, {\rm and}\,\,  \limsup_{t\to\infty} \sup_{x\in \Omega} v(t,x;u_0,v_0)=0.
$$
We will study the persistence of both species and competitive exclusion somewhere else.
\end{itemize}
\end{remark}

{ We conclude the introduction with some remarks on the following full parabolic counterpart of \eqref{main-eq},
\begin{equation}
\label{full-parabolic-eq}
\begin{cases}
u_t=\Delta u-\chi_1 \nabla\cdot (\frac{u}{w} \nabla w)+u(a_1-b_1u-c_1v) ,\quad &x\in \Omega\cr
v_t=\Delta v-\chi_2 \nabla\cdot (\frac{v}{w} \nabla w)+v(a_2-b_2v-c_2u),\quad &x\in \Omega\cr
w_t=\Delta w-\mu w +\nu u+ \lambda v,\quad &x\in \Omega \cr
\frac{\p u}{\p n}=\frac{\p v}{\p n}=\frac{\p w}{\p n}=0,\quad &x\in\p\Omega.
\end{cases}
\end{equation}
For given  $u_0,v_0,w_0$ satisfying
\begin{equation}
\label{full-parabolic-initial-cond-eq}
u_0,v_0(\cdot)\in C^0(\bar\Omega), \,\, w_0\in W^{1,\infty}(\Omega),\,\,
u_0\ge 0,\,\, v_0\ge 0,\,\, \int_\Omega (u_0(x)+v_0(x))dx>0,\,\,
w_0(x)>0,
\end{equation}
 by standard contraction arguments (see \cite{bel-wi, FuWiYo1, HoWi}), there exists $T_{\max}(u_0,v_0,w_0)\in (0,\infty]$
such that  \eqref{full-parabolic-eq} possesses a unique  positive classical solution, denoted by $(u(t,x;u_0,v_0,w_0)$, $v(t,x;u_0,v_0,w_0)$, $w(t,x;u_0,v_0,w_0))$,  on $(0,T_{\max}(u_0,v_0,w_0))$ with initial condition $(u(0,x;u_0,v_0,w_0)$ , $v(0,x;u_0,v_0,w_0)$, $w(0,x;u_0,v_0,w_0))=(u_0(x),v_0(x),w_0(x))$.
Moreover if $T_{\max}(u_0,v_0,w_0)< \infty,$ then either
\begin{equation*}
\label{local-infty-1-0}
\limsup_{t \nearrow T_{\max}(u_0,v_0,w_0)} \left(\left\| u(t,\cdot;u_0,v_0,w_0)+v(t,\cdot;u_0,v_0,w_0) \right\|_{C^0(\bar \Omega)} +\left\|w(t,\cdot;u_0,v_0,w_0)\right\|_{W^{1,\infty}(\Omega)}\right) =\infty,
\end{equation*}
or
\begin{equation*}
\label{local-infty-2-0}
    \liminf_{t \nearrow T_{\max}(u_0,v_0,w_0)} \inf_{x \in \Omega} w(t,x;u_0,v_0,w_0) =0.
\end{equation*}

Observe  that,  for given $u_0,v_0,w_0$ satisfying \eqref{full-parabolic-initial-cond-eq}, if $v_0(x)\equiv 0$, then $v(t,x;u_0,v_0,w_0)\equiv 0$ for $t\in (0,T_{\max}(u_0,v_0,w_0))$ and $(u(t,x;u_0,w_0),w(t,x;u_0,w_0)):=(u(t,x;u_0,v_0,w_0),w(t,x;u_0,v_0,w_0))$ is the classical solution of the following full parabolic counterpart of \eqref{main-eq1}
\begin{equation}
\label{full-parabolic-eq1}
\begin{cases}
u_t=\Delta u-\chi_1 \nabla\cdot (\frac{u}{w} \nabla w)+u(a_1-b_1u) ,\quad &x\in \Omega\cr
w_t=\Delta w-\mu w +\nu u,\quad &x\in \Omega \cr
\frac{\p u}{\p n}=\frac{\p w}{\p n}=0,\quad &x\in\p\Omega
\end{cases}
\end{equation}
with initial condition $(u(0,x;u_0,w_0),w(0,x;u_0,w_0))=(u_0(x),w_0(x))$.
Similarly, if $u_0\equiv 0$, then $u(t,x;u_0,v_0,w_0)\equiv 0$ on $(0,T_{\max}(u_0,v_0,w_0))$ and $(v(t,x;v_0,w_0),w(t,x;v_0,w_0)):=(v(t,x;u_0,v_0,w_0)$, $w(t,x;u_0,v_0,w_0))$ is the classical solution of
\begin{equation}
\label{full-parabolic-eq2}
\begin{cases}
v_t=\Delta v-\chi_2 \nabla\cdot (\frac{v}{w} \nabla w)+v(a_2-b_2v),\quad &x\in \Omega\cr
w_t=\Delta w-\mu w + \lambda v,\quad &x\in \Omega \cr
\frac{\p v}{\p n}=\frac{\p w}{\p n}=0,\quad &x\in\p\Omega
\end{cases}\,
\end{equation}
with initial condition $(v(0,x;v_0,w_0),w(0,x;v_0,w_0))=(v_0(x),w_0(x))$.

Systems \eqref{full-parabolic-eq1} and \eqref{full-parabolic-eq2} are essentially the same.  There are several works on the global existence and asymptotic behavior of classical or weak solutions of \eqref{full-parabolic-eq1}. For example, in the case  $N=2$, the authors of \cite{AiOs} proved the global existence of solutions of \eqref{full-parabolic-eq1} with initial functions $u_0\in L^2(\Omega)$, $u_0\ge 0$, and $w_0\in H^{1+\theta_0}(\Omega)$ for some $\theta_0\in (0,1/2)$, $\inf_{x\in\Omega}w_0(x)>0$ (see \cite[Theorem 2.1]{AiOs});  the authors of \cite{ZhZh} proved the global existence and boundedness of classical solutions of \eqref{full-parabolic-eq1} with initial functions $u_0\in C^0(\bar \Omega)$, $u_0(x)\ge 0$, $u_0\not\equiv 0$, and $w_0\in W^{1,q}(\Omega)$ for some $q>2$, $w_0(x)>0$ provided that $\chi_1$ is  relatively small with respect  to $a_1$  (see \cite[Theorem 1]{ZhZh});  the authors of \cite{ZhMu} showed the global stability of the positive constant solution of \eqref{full-parabolic-eq1} provided that $\chi_1$ is relatively small with respect  to $a_1$ (see \cite[Theorem 1.1]{ZhMu}).
For general $N\ge 1$,  global existence of weak solutions of \eqref{full-parabolic-eq1} is studied in \cite{DiWaZh,ZhZh1}. The authors of \cite{Fuj,Lan, Win} studied the  global existence and boundedness of classical solutions of \eqref{full-parabolic-eq1} with $a_1=b_1=0$.

Up to our knowledge, it  remains open whether \eqref{full-parabolic-eq1} has a global classical solution for any given initial functions $u_0\in C^0(\bar \Omega)$, $u_0(x)\ge 0$, $u_0\not\equiv 0$, and $w_0\in W^{1,\infty}(\Omega)$, $w_0(x)>0$ in any space dimensional setting.
There is little study on the global existence and boundedness of classical solutions of \eqref{full-parabolic-eq} for any given initial functions   $u_0,v_0,w_0$ satisfying \eqref{full-parabolic-initial-cond-eq}.  We remark that finite upper bounds of $\int_\Omega (u(t,x)+v(t,x))^pdx$ for some $p\gg 1$ and positive lower bounds of $w(t,x)$, or equivalently, finite upper bounds of
$\int_\Omega(u(t,x)+v(t,x))^{-q}$ for some $q>0$,  are among the key ingredients in  the proofs of global existence and boundedness of the classical solution $u(t,x),v(t,x),w(t,x))$ of \eqref{main-eq} with initial functions $u_0,v_0$ satisfying \eqref{initial-cond-eq}.
We expect that such bounds for the solutions of \eqref{full-parabolic-eq} if they can be obtained will also ensure the global existence and boundedness of classical solutions of \eqref{full-parabolic-eq} with initial conditions $u_0,v_0,w_0$ satisfying \eqref{full-parabolic-initial-cond-eq}.
However, the techniques in the current paper to obtain
the upper bounds of $\int_\Omega(u+v)^p$ for $p\gg 1$ and $\int_\Omega(u+v)^{-q}$ for some $q>0$ for the solutions of \eqref{main-eq}
rely on the fact that the equation for $w$ in \eqref{main-eq} is  elliptic (see Remarks \ref{elliptic-rk1} and \ref{elliptic-rk2}).
New techniques/methods need to be developed to get such bounds for the solutions of \eqref{full-parabolic-eq}. We wish to carry out some study on   the global existence and various properties of classical solutions of \eqref{full-parabolic-eq} in the near future.
}

The rest of the paper is organized as follows.  In section 2,  we present some preliminary lemmas.  In section 3, we prove Theorems \ref{new-main-thm1} and
and \ref{new-main-thm4}. Section 4 is devoted to the proof of Theorem \ref{new-main-thm2}.
Theorems \ref{new-main-thm3} and \ref{new-main-thm5} are proved in section 5.
 We prove one important technical proposition in the Appendix.

\section{Preliminary lemmas}

In this section, we present some lemmas to be used in later sections.

First, let $\chi^*(\mu,\chi_1,\chi_2)$ be defined as in \eqref{bound-for-a-eq3}. We  present the following lemma on the continuity and {upper bounds}  of $\chi^*(\mu,\chi_1,\chi_2)$.

\begin{lemma}
\label{chi-star-lm}
\begin{itemize}
\item[(1)]
$\chi^*(\mu,\chi_1,\chi_2)$ is  upper semicontinuous in $\mu>0$, $\chi_1>0$ and $\chi_2>0$, that is,
$$
\limsup_{(\mu,\chi_1,\chi_2)\to (\mu^0,\chi_1^0,\chi_2^0)}\chi^*(\mu,\chi_1,\chi_2)\le \chi^*(\mu^0,\chi_1^0,\chi_2^0)\quad\forall\, \mu^0>0,\,\,\chi_1^0>0,\,\, \chi_2^0>0.
$$

\item[(2)]
For any $\mu>0$, $\chi_1>0$ and $\chi_2>0$,
\begin{equation}
\label{main-assumption-1}
    \chi^*(\mu,\chi_1,\chi_2)\le \min\Big\{\mu \chi_2+\frac{\mu (\chi_1-\chi_2)^2}{{ 4}},\mu\chi_1+\frac {\mu(\chi_2-\chi_1)^2} {{ 4}}\Big\},
\end{equation}
and when $\chi_1=\chi_2=\chi$,
\begin{equation}
\label{main-assumption-2}
\chi^*(\mu,\chi,\chi){ \le} \begin{cases} \frac{\mu \chi^2}{4} &{\rm if}\,\, 0<\chi<2\cr
\mu(\chi-1) &{\rm if}\,\, \chi \ge 2.
\end{cases}
\end{equation}
\end{itemize}
\end{lemma}

{\begin{remark}
\label{chi-star-rk}
We point out that the boundedness of $\int_\Omega (u(t,x;u_0,v_0)+v(t,x;u_0,v_0))^{-q}dx$ for some $q>0$ plays a crucial role in the study of the boundedness of $u(t,x;u_0,v_0)+v(t,x;u_0,v_0)$. The assumption that $a_{\min}>\chi^*(\mu,\chi_1,\chi_2)$ indicates that
the chemotaxis sensitivities $\chi_1,\chi_2$ and the degradation rate $\mu$ of the chemical substance are small relatively with respect to $a_{\min}$ and ensures the boundedness
 of
$\int_\Omega (u(t,x;u_0,v_0)+v(t,x;u_0,v_0))^{-q}dx$ for some $q>0$ (see Lemma \ref{mass-persistence-thm} and the arguments of Theorem \ref{new-main-thm2}). In the following, we provide some discussions on some specific upper bounds of $\chi^*(\mu,\chi_1,\chi_2)$   for the reader to get some comprehensive
feeling about the impact of $\chi_1,\chi_2$ and $\mu$ on the boundedness of
$\int_\Omega(u+v)^{-q}$. First, note that
\begin{align}
\label{new-u+v-eq1-0}
\frac{\p }{\p t}(u+v)=& \Delta(u+v)-\nabla\cdot \left(\frac{\chi_1 u+\chi_2 v}{w}\nabla w\right)\nonumber\\
& +(a_1u+a_2 v)-(b_1u^2+b_2 v^2)-(c_1+c_2)u v .
\end{align}
Next, for any $q>0$,
multiplying  \eqref{new-u+v-eq1-0} by  $(u+v)^{-q-1}$ and integrating over $\Omega$, we have
\begin{align*}
    \frac{1}{q}\frac{d}{dt}\int_\Omega (u+v)^{-q}
=& -(q+1)\int_\Omega (u+v)^{-q-2}|\nabla (u+v)|^2\nonumber\\
& +(q+1)\int_\Omega (u+v)^{-q-2}\frac{\chi_1 u+\chi_2 v}{w} \nabla(u+v)\cdot\nabla w\nonumber\\
& -\int_\Omega (u+v)^{-q-1}(a_1u+a_2 v)+\int_\Omega (u+v)^{-q-1}(b_1 u^2+b_2 v^2)\nonumber\\
&+\int_\Omega (u+v)^{-q-1}(c_1+c_2) uv
\end{align*}
\begin{align}
\label{new-LP-eq0-0}
\hspace{0.5in}
\le &- (q+1)\int_\Omega (u+v)^{-q-2}|\nabla (u+v)|^2\nonumber\\
&+(q+1)\chi_2 \int_\Omega
\frac{(u+v)^{-q-1}}{w}\nabla (u+v)\cdot\nabla w\nonumber\\
& +(q+1)(\chi_1-\chi_2)\int_\Omega (u+v)^{-q-2}\frac{u}{w}\nabla (u+v)\cdot\nabla w\nonumber\\
& - a_{\min} \int_\Omega (u+v)^{-q}+b_{\max} \int_\Omega (u+v)^{-q+1}+c_{\max} \int_\Omega (u+v)^{-q+1}\nonumber\\
\le &- (q+1)\int_\Omega (u+v)^{-q-2}|\nabla (u+v)|^2\nonumber\\
&+(q+1)\chi_2 \int_\Omega
\frac{(u+v)^{-q-1}}{w}\nabla (u+v)\cdot\nabla w\nonumber\\
&+\frac{(\chi_1-\chi_2)^2 q(q+1)}{4B}\int_\Omega
(u+v)^{-q-2}|\nabla (u+v)|^2 +\frac{B(q+1)}{q}\int_\Omega (u+v)^{-q}\frac{|\nabla w|^2}{w^2}\nonumber\\
& - a_{\min} \int_\Omega (u+v)^{-q}+(b_{\max}+c_{\max}) \int_\Omega (u+v)^{-q+1}, \,\,\, \,\, \forall\, B>0.
\end{align}
 Now, by the third equation in \eqref{main-eq}, we have
$$
\frac{B}{q}\int_\Omega (u+v)^{-q}\frac{|\nabla w|^2}{w^2}\le \frac{B\mu}{q}\int_\Omega (u+v)^{-q}-B\int_\Omega \frac{(u+v)^{-q-1}}{w}\nabla(u+v)\cdot\nabla w
$$
(see the arguments of Lemma \ref{mass-persistence-thm} for detail).
Hence
\begin{align*}
\frac{1}{q}\frac{d}{dt}\int_\Omega(u+v)^{-q} \le &- (q+1)\int_\Omega (u+v)^{-q-2}|\nabla (u+v)|^2\nonumber\\
&+(q+1)\chi_2 \int_\Omega
\frac{(u+v)^{-q-1}}{w}\nabla (u+v)\cdot\nabla w\nonumber\\
&+   \frac{(\chi_1-\chi_2)^2 q(q+1)}{4B}\int_\Omega
(u+v)^{-q-2}|\nabla (u+v)|^2 \nonumber\\
&+ \frac{B(q+1)\mu}{q}\int_\Omega (u+v)^{-q}-B(q+1)\int_\Omega \frac{(u+v)^{-q-1}}{w}\nabla(u+v)\cdot\nabla w\nonumber\\
& - a_{\min} \int_\Omega (u+v)^{-q}+(b_{\max}+c_{\max}) \int_\Omega (u+v)^{-q+1}
\end{align*}
for any $q>0, B>0$. Let $B=\chi_2$ and $q=\frac{4\chi_2}{(\chi_1-\chi_2)^2}$ in the case $\chi_1\not =\chi_2$. We get
$$
\frac{1}{q}\frac{d}{dt}\int_\Omega(u+v)^{-q}\le -\Big(a_{\min}-\mu \big(\chi_2+\frac{(\chi_1-\chi_2)^2}{4}\big)\Big)\int_\Omega(u+v)^{-q}+(b_{\max}+c_{\max}) \int_\Omega (u+v)^{-q+1}.
$$
Therefore, $a_{\min}>\mu\big(\chi_2+\frac{(\chi_1-\chi_2)^2}{4}\big)$ ensures the boundedness of
$\int_\Omega(u+v)^{-q}$  (see the arguments of Theorem \ref{new-main-thm2}) and  $\mu \big(\chi_2+\frac{(\chi_1-\chi_2)^2}{4}\big)$ is an upper bound of $\chi^*(\mu,\chi_1,\chi_2)$. Similarly,   $\mu\big(\chi_1+\frac{(\chi_1-\chi_2)^2}{4}\big)$ is an upper bound of $\chi^*(\mu,\chi_1,\chi_2)$.
\end{remark}
}

\newpage
\begin{proof}[Proof of Lemma \ref{chi-star-lm}]
(1) We prove that $\chi^*(\mu,\chi_1,\chi_2)$ is upper semicontinuous in   $\mu>0$, $\chi_1>0$ and $\chi_2>0$.

Fix $(\mu^0,\chi_1^0, \chi_2^0)$ with $\mu^0, \chi_1^0, \chi_2^0>0.$
Without loss of generality, we may assume that $\chi^*(\mu^0,\chi_1^0,\chi_2^0)=\chi^*_1(\mu^0,\chi_1^0,\chi_2^0).$  Then, for any $\epsilon>0,$ there are $\tilde\beta>0$ and $\tilde B>0$ satisfying $\tilde\beta\not ={\chi_2^0-\tilde B}$ such that
\begin{equation*}
    \chi^*_1(\mu^0,\chi_1^0, \chi_2^0) \ge f(\mu^0,\chi_1^0, \chi_2^0,\tilde \beta, \tilde B)-\epsilon ,
\end{equation*}
 and there is $\delta>0$ such that
$$
\tilde\beta\not ={ \chi_2-\tilde B}
$$
and
\begin{equation*}
    f(\mu,\chi_1, \chi_2,\tilde \beta, \tilde B) \leq f(\mu^0,\chi_1^0, \chi_2^0,\tilde \beta, \tilde B)+\epsilon
\end{equation*}
 for all $\mu>0,\chi_1>0,\chi_2>0$ satisfying $|\mu-\mu^0|<\delta$, $|\chi_1-\chi_1^0|<\delta$, and $|\chi_2-\chi_2^0|<\delta$.  This implies that
\begin{equation*}
    f(\mu,\chi_1, \chi_2,\tilde \beta, \tilde B)  \leq f(\mu^0,\chi_1^0, \chi_2^0,\tilde \beta, \tilde B)+\epsilon \leq \chi^*_1(\mu^0,\chi_1^0, \chi_2^0) +2\epsilon,
\end{equation*}
and then
\begin{equation*}
    \chi_1^*(\mu,\chi_1,\chi_2) \leq \chi^*_1(\mu^0,\chi_1^0, \chi_2^0) +2\epsilon.
\end{equation*}
Therefore,
$$
\limsup_{(\mu,\chi_1,\chi_2)\to (\mu^0,\chi_1^0,\chi_2^0)}\chi^*(\mu,\chi_1,\chi_2)\le \chi^*(\mu^0,\chi_1^0,\chi_2^0)\quad\forall\, \mu^0,\chi_1^0, \chi_2^0>0.
$$

(2)  We first prove  \eqref{main-assumption-2}
when  $\chi_1=\chi_2$. In this case,  we have
$$
f(\mu,\chi_1,\chi_2, \beta, B)=\mu(B+\beta) \big(1+\frac{({ \chi_2-B} -\beta)^2}{4\beta}\big).
$$
This implies that
\begin{align*}
\chi_1^*(\mu,\chi_1,\chi_2)&=\inf\Big\{  \mu(B+\beta) \big(1+\frac{({ \chi_2-B} -\beta)^2}{4\beta}\big)\,|\, B>0,\, \beta>0, \beta\not ={ \chi_2-B}\Big\}\nonumber\\
&\le \inf\Big\{\mu \beta \Big(1+\frac{(\chi_2-\beta)^2}{4\beta}\Big)\,|\, 0< \beta<\chi_2\Big\}\nonumber\\
&=\begin{cases} \frac{\mu \chi_2^2}{4}\quad &{\rm if}\,\, 0<\chi_2<2\cr
\mu(\chi_2-1)\quad &{\rm if}\,\, \chi_2>2.
\end{cases}
\end{align*}
Similarly,
$$
\chi_2^*(\mu,\chi_1,\chi_2)\le \begin{cases} \frac{\mu \chi_{1}^2}{4}\quad &{\rm if}\,\, 0<\chi_{ 1}<2\cr
\mu(\chi_{ 1}-1)\quad &{\rm if}\,\, \chi_{ 1}>2.
\end{cases}
$$
Hence \eqref{main-assumption-2} holds.

Next  we prove \eqref{main-assumption-1} for any $\mu>0$, $\chi_1>0$ and $\chi_2>0$.
Recall that
$$
f(\mu,\chi_1,\chi_2, \beta, B)=\mu(B+\beta) \Big(1+\frac{{ B}({\chi_2-B} -\beta)^2+(\chi_1-\chi_2)^2\beta}{4B\beta}\Big).
$$
Then
$$
f(\mu,\chi_1,\chi_2,\beta,\chi_2)=\mu(\chi_2+\beta)\Big(1+\frac{\beta\chi_2+(x_1-x_2)^2}{4\chi_2}\Big)
$$
and \vspace{-0.05in}
\begin{align*}
\chi_1^*(\mu,\chi_1,\chi_2)\le \inf\Big\{f(\mu,\chi_1,\chi_2,\beta,\chi_2)\,|\, \beta>0\Big\}=\mu \chi_2+\frac{\mu (\chi_1-\chi_2)^2}{{ 4}}.
\end{align*}
Similarly,
\begin{align*}
\chi_2^*(\mu,\chi_1,\chi_2)\le \inf\Big\{f(\mu,\chi_2,\chi_1,\beta,\chi_1)\,|\, \beta>0\Big\}=\mu \chi_1+\frac{\mu (\chi_2-\chi_1)^2}{{ 4}}.
\end{align*}
Hence \eqref{main-assumption-1} holds.
\end{proof}

Next, we present some properties of the semigroup generated by $-\Delta+ \mu I$ complemented with Neumann boundary condition on $L^p(\Omega)$.
For given $1< p<\infty$,  let $X_p=L^{p}(\Omega)$
  and
\begin{equation}
\label{A-p-def}
 A_p=-\Delta+ \mu I: D(A_p)\subset L^p(\Omega)\to L^p(\Omega)
\end{equation}
 with
  $$D(A_p)=\left\{ u \in W^{2,p}(\Omega) \, |  \, \frac{\p u}{\p n}=0 \quad \text{on } \, \p \Omega \right\}.
  $$   Then $-A_p$ generates an analytic semigroup on $L^p(\Omega)$. We denote it by
$e^{-tA_p}$. Note that ${\rm Re}\sigma(A_p)>0$.  Let $A_p^\beta$ be the fractional power operator of $A_p$ (see \cite[Definition 1.4.1]{Hen}).
Let $X_p^\beta=\mathcal{D}(A_p^\beta)$ with graph norm $\|u\|_{X_p^\beta}=\|A_p^\beta u\|_{L^p(\Omega)}$ for $\beta\ge 0$ and $u\in X_p^\beta$
(see \cite[Definition 1.4.7]{Hen}).

\begin{lemma}
\label{pre-lm-2}
\begin{itemize}
\item[(i)] For each  $p\in (1,\infty)$ and $\beta\ge 0$, there is $C_{p,\beta}>0$ such that for some $\gamma>0$,
\begin{equation*}
    \|A_p^\beta e^{-A_pt}\|_{L^p(\Omega)} \leq C_{p,\beta} t^{-\beta} e^{-\gamma t}  \quad \, for \; t>0.
\end{equation*}

\item[(ii)]
If $m \in \{0,1\}$ and $q\in [p,\infty]$ are such that $m-\frac{N}{q}<2\beta-\frac{N}{p}$,
then
$
 X_p^\beta\hookrightarrow W^{m,q}(\Omega).
$

\item[(iii)] If $2\beta -\frac{N}{p}>\theta\ge 0$, then
$
X_p^\beta\hookrightarrow C^\theta(\Omega).
$
\end{itemize}
\end{lemma}

\begin{proof}
(i) It follows from \cite[Theorem 1.4.3]{Hen}.

(ii) It follows from \cite[Theorem 1.6.1]{Hen}.

(iii) It also follows from \cite[Theorem 1.6.1]{Hen}.
\end{proof}

\begin{lemma}
\label{pre-lm-3}
Let $\beta \geq 0,$ $p \in (1,\infty)$. Then for any $\epsilon >0$ there exists $C_{p,\beta,\epsilon}>0$ such that for any $w \in  C^{\infty}_0(\Omega)$ we have
\begin{equation}
\label{001}
\|A_p^{\beta}e^{-tA_p}\nabla\cdot w\|_{L^p(\Omega)}  \leq C_{p,\beta,\epsilon} t^{-\beta-\frac{1}{2}-\epsilon} e^{-\gamma  t} \|w\|_{L^p(\Omega)} \quad \text{for all}\,\,  t>0  \, \text{and  some } \, \gamma>0.
\end{equation}
Consequently,  for all $t>0$ the operator $A_p^\beta e^{-tA_p}\nabla\cdot $ admits a unique extension to all of $L^p(\Omega)$ which is again denoted by $A_p^\beta e^{-t A_p}\nabla\cdot$  and  satisfies $\eqref{001}$ for all $\mathbb{R}^N$-valued $w \in L^p(\Omega).$
\end{lemma}

\begin{proof}
It follows from \cite[Lemma 2.1]{HoWi}.
\end{proof}

\begin{remark}
\label{fractional-power-space-rk}
We remark that  the fractional powers of $A_p$ are used to prove the boundedness of the $C^\theta(\bar\Omega)$-norm of the solutions of \eqref{main-eq} for some $\theta>0$.
Such techniques are also used in \cite{HoWi} (see \cite[Theorem 4.1]{HoWi}).
It should be pointed out that we can prove the boundedness of the $L^p(\Omega)$-norm of
the solutions of \eqref{main-eq} for any $p\ge 1$ via the standard estimates for the Neumann heat semigroup  $\{e^{t\Delta}\}_{t\ge 0}$ in $\Omega$
such as
$$
\|e^{t\Delta}u\|_{L^q(\Omega)}\le C \Big(1+t^{-\frac{N}{2}\big(\frac{1}{p}-\frac{1}{q}\big)}\Big)\|u\|_{L^p(\Omega)},
$$
$$
\|\nabla e^{t\Delta}u\|_{L^q(\Omega)}\le C \Big(1+t^{-\frac{1}{2}-\frac{N}{2}\big(\frac{1}{p}-\frac{1}{q}\big)}\Big) e^{-\lambda_1 t}\|u\|_{L^p(\Omega)},
$$
for some $C>0$, all $t>0$, $u\in L^p(\Omega)$, and $q>p$ (see \cite[Lemma 1.3]{Win} for some  other estimates for the Neumann heat semigroup $\{e^{t\Delta}\}_{t\ge 0}$).
Due to the presence of the chemotaxis in \eqref{main-eq}, some estimates for the fractional powers of $A_p$ or some  other arguments such as bootstrap arguments are needed to get the $C^\theta(\bar\Omega)$-boundedness of the solutions of \eqref{main-eq} for some $\theta>0$.
\end{remark}

We then present   some lemmas on the lower and upper bounds of  the solutions of
 \begin{equation}
\label{w-eq}
    \begin{cases}
    -\Delta w +\mu w=\nu u +\lambda v, \quad &x\in \Omega \cr
    \frac{\p w}{\p n}=0, \quad & x\in \p \Omega,
    \end{cases}
\end{equation}
where $\mu,\nu$, and $\lambda$ are positive constants.
For given $u,v\in L^p(\Omega)$, let $w(\cdot;u,v)$ be the solution of \eqref{w-eq}.

\begin{lemma}
\label{pre-lm-1}
Let $u,v \in C^0(\bar \Omega)$ be nonnegative function such that $\int_{\Omega} u >0$ and $\int_{\Omega} v >0$.
Then
\begin{equation*}
    w (x;u,v)\ge \delta_0  \int_{\Omega} (u + v) >0 \quad {\rm in} \quad \Omega,
\end{equation*}
where $\delta_0$ is some positive constant independent of $u,v$.
\end{lemma}

\begin{proof}
It follows from the arguments of  \cite[Lemma 2.1]{FuWiYo} and the
Gaussian lower bound for the heat kernel of the laplacian with Neumann boundary
condition on smooth domain (see \cite[Theorem 4]{ChOuYa}).
\end{proof}

\begin{lemma}
\label{pre-lm-4}
For any $p\ge 1$, there exists $C_p>0$ such that
\begin{equation*}
 \max\Big\{\|w(\cdot;u,v)\|_{L^p(\Omega)},   \| \nabla w(\cdot;u,v) \|_{L^p(\Omega)}\Big\} \leq C_p  \|u(\cdot)+v(\cdot)\|_{L^p(\Omega)}\quad \forall u,v\in L^p(\Omega).
\end{equation*}
\end{lemma}

\begin{proof}
It follows from $L^p$-estimates for elliptic equations (see \cite[Theorem 12.1]{Ama}).
\end{proof}

\begin{lemma}
\label{pre-lm-6}
For any nonnegative $u,v\in C(\bar\Omega)$,
\begin{equation*}
    \int_{\Omega}\frac{|\nabla w(\cdot;u,v)|^2}{w(\cdot;u,v)^2} \leq \mu |\Omega|.
\end{equation*}
\end{lemma}

\begin{proof}
Multiplying \eqref{w-eq} by $\frac{1}{w}$ and integrating it over $\Omega$ yields that
\begin{align*}
    0= \int_{\Omega} \frac{1}{w} \cdot \Big( \Delta w -\mu w + \nu u + \lambda v\Big)
    = \int_{\Omega}\frac{|\nabla w|^2}{w^2} - \mu |\Omega| + \nu \int_{\Omega}\frac{u}{w} +\lambda \int_{\Omega}\frac{v}{w}
\end{align*}
for  all $t \in (0,T_{\max})$. The lemma thus follows.
\end{proof}

Throughout the rest of this section, we assume that $u_0(x)$ and $v_0(x)$ satisfies \eqref{initial-cond-eq} and
$(u(t,x),v(t,x),w(t,x)):=(u(t,x;u_0,v_0),v(t,x;u_0,v_0),w(t,x;u_0,v_0))$ is the unique classical solution of \eqref{main-eq} on the maximal interval $(0,T_{\max}):=(0,T_{\max}(u_0,v_0))$ with the initial condition $(u(0,x),v(0,x))=(u_0(x), v_0(x))$.
Note that
$$
u(t,x), v(t,x), w(t,x)>0\quad \forall\, x\in\Omega,\,\, t\in (0,T_{\max}).
$$
We may drop $(t,x)$ in $u(t,x)$, $v(t,x)$, and $w(t,x)$ if no confusion occurs.

We now present some upper bound for $\int_\Omega u(t,x;u_0,v_0)dx$ and $\int_\Omega v(t,x;u_0,v_0)dx$.

\begin{lemma}
\label{pre-lm-5} For any $\tau\in [0,T_{\max})$, { the followings hold.}
    \begin{equation}
\label{new-u-upper-bound-eq}
    \int_{\Omega} u (t,x;u_0,v_0)dx \leq m^*_1(\tau,u_0,v_0):= {\rm max}\Big\{\int_{\Omega} u(\tau,x;u_0,v_0)dx,  \frac{a_{1}|\Omega|}{b_{1}} \Big\},
\end{equation}
{ and}
\begin{equation}
\label{new-v-upper-bound-eq}
    \int_{\Omega} v (t,x;u_0,v_0)dx \leq m^*_2(\tau,u_0,v_0):= {\rm max}\Big\{\int_{\Omega} v(\tau,x;u_0,v_0)dx,  \frac{a_{2}|\Omega|}{b_{2}} \Big\},
\end{equation}
{ as well as}
\begin{align}
\label{new-u+v-upper-bound-eq}
    & \int_{\Omega} \big(u(t,x;u_0,v_0)+v (t,x;u_0,v_0)\big)dx \nonumber\\
&\leq m^*(\tau,u_0,v_0):= {\rm max}\Big\{\int_{\Omega}(u(\tau,x;u_0,v_0)+ v(\tau,x;u_0,v_0))dx,  \frac{a_{\max}|\Omega|}{\min\{b_{\min},c_{\min}\}} \Big\}
\end{align}
for any $t\in [\tau,T_{\max})$,
where $|\Omega|$ is the Lebesgue measure of $\Omega$.
Moreover, if $T_{\max}(u_0,v_0)=\infty$, then
\begin{equation}
\label{new-u-v-upper-bd-eq}
\limsup_{t\to\infty}\int_\Omega u(t,x;u_0,v_0)dx\le \frac{a_1|\Omega|}{b_1}, \quad \limsup_{t\to\infty} \int_\Omega v(t,x;u_0,v_0)dx\le \frac{a_2|\Omega|}{b_2},
\end{equation}
and
\begin{equation}
\label{new-u-v-upper-bd-eq1}
\limsup_{t\to\infty}\int_\Omega ( u(t,x;u_0,v_0)+v(t,x;u_0,v_0))dx\le  \frac{a_{\max}|\Omega|}{\min\{b_{\min},c_{\min}\}}.
\end{equation}
\end{lemma}

\begin{proof}
By integrating the first equation in \eqref{main-eq} with respect to $x$, we get that
\begin{align*}
\frac{d}{dt} \int_{\Omega}u&=\int_{\Omega}\Delta u- \chi  \int_{\Omega} \nabla \cdot \Big(\frac{u}{w}\nabla w \Big)+a_1 \int_{\Omega}  u - b_1 \int_{\Omega}  u^2 - c_1\int_{\Omega}  v u\\
    &  \leq a_{1} \int_{\Omega}u - \frac{b_{1}}{|\Omega|} \Big(\int_{\Omega} u\Big)^2 \quad \forall \, t\in (0,T_{\max}).
\end{align*}
This together with  comparison principle for scalar ODEs implies \eqref{new-u-upper-bound-eq} and the first inequality in \eqref{new-u-v-upper-bd-eq} when $T_{\max}=\infty$.

Similarly, we can prove \eqref{new-v-upper-bound-eq} and the second inequality in \eqref{new-u-v-upper-bd-eq} when $T_{\max}=\infty$.

{ In the following, we prove \eqref{new-u+v-upper-bound-eq}.
By \eqref{new-u+v-eq1-0},
\begin{align*}
\frac{d}{d t}\int_\Omega (u+v)&=\int_\Omega (a_1u+a_2 v)-\int_\Omega (b_1u^2+b_2 v^2)-\int_\Omega (c_1+c_2)u v\\
&\le a_{\max}\int_\Omega (u+v)-b_{\min} \int_\Omega(u^2+v^2)-2c_{\min}\int_\Omega uv\\
&=a_{\min}\int_\Omega (u+v)-\min\{b_{\min},c_{\min}\}\int_\Omega (u+v)^2\\
&\le a_{\max}\int_\Omega (u+v)-\frac{\min\{b_{\min},c_{\min}\}}{|\Omega|}\Big(\int_\Omega(u+v)\Big)^2,
\end{align*}
for all $t \in (0,T_{\max}).$  This together with  comparison principle for scalar ODEs implies \eqref{new-u+v-upper-bound-eq} and  \eqref{new-u-v-upper-bd-eq1} when $T_{\max}(u_0,v_0)=\infty$.}
\end{proof}

{
\begin{lemma}
\label{new-pre-lm1}
For any $\varepsilon>0$ and $p>0$, there is $C(\varepsilon,p)>0$ such that
$$
\int_\Omega w^{p+1}\le\varepsilon \int_\Omega (u+v)^{p+1}+C(\varepsilon,p)\Big(\int_\Omega (u+v)\Big)^{p+1}  \quad\text{for all}\;\; t \in (0,T_{\max}).
$$
\end{lemma}}

{\begin{proof}
First, multiplying the third equation in \eqref{main-eq} by $w^{p}$ and integrating over $\Omega$, along with applying Young's inequality, we have
\begin{align*}
    p\int_{\Omega} w^{p-1}|\nabla w|^2 + \mu \int_{\Omega}w^{p+1} &\leq \max\{\nu,\lambda\}  \int_{\Omega}(u+v) w^{p}\\
    &\leq  \left(\frac{\max\{\nu,\lambda\}}{p+1}\right)^{p+1}\left(\frac{p}{\mu}\right)^p\int_{\Omega} (u+v)^{p+1} +\mu \int_{\Omega} w^{p+1},
\end{align*}
which implies
\begin{equation}
\label{aux-new-eq1*}
    \int_{\Omega} w^{p-1}|\nabla w|^2  \leq \frac{1}{p} \left(\frac{\max\{\nu,\lambda\}}{p+1}\right)^{p+1}\left(\frac{p}{\mu}\right)^p\int_{\Omega} (u+v)^{p+1} \quad\text{for all}\;\; t \in (0,T_{\max}).
\end{equation}
Then, for any  $\tilde \varepsilon>0,$ there are $ C_1,\tilde C_1>0$ such that

 \begin{align*}
  &  \int_{\Omega} w^{p+1}=\int_{\Omega} { (w^{\frac{p+1}{2}})^2}\\
&\le \tilde\varepsilon \int_{\Omega} |\nabla w^{\frac{p+1}{2}}|^2+C _1 \left(\int_{\Omega}  w^{\frac{p+1}{2}} \right)^{2}\quad \quad \quad \quad \quad \quad\quad \quad \quad \quad\qquad\qquad\quad\quad \quad \quad \quad \, \, \;\;\text{(by Ehrling's lemma)}\nonumber\\
   &\le \frac{\tilde\varepsilon (p+1)^2}{4}\int_{\Omega} w^{p-1}|\nabla w|^2 +
  C_1\Big( \int_\Omega w^{p}\Big)\Big(\int_\Omega w\Big) \quad \quad \quad \quad \quad \quad\quad \quad \quad\quad\quad \quad \quad \quad \text{(by H\"older's ineq.)}\nonumber\\
&\le \frac{\tilde\varepsilon (p+1)^2}{4p}\left(\frac{\max\{\nu,\lambda\}}{p+1}\right)^{p+1}\left(\frac{p}{\mu}\right)^p\int_{\Omega} (u+v)^{p+1} +
  C_1\Big( \int_\Omega w^{p}\Big)\Big(\int_\Omega w\Big) \quad \quad \quad \quad \text{(by \eqref{aux-new-eq1*})}\nonumber\\
   &\le  \frac{\tilde\varepsilon (p+1)^2}{4p}\left(\frac{\max\{\nu,\lambda\}}{p+1}\right)^{p+1}\left(\frac{p}{\mu}\right)^p\int_{\Omega} (u+v)^{p+1} +
   \frac{1}{2} \int_\Omega w^{p+1}+\tilde C_1\Big(\int_\Omega w\Big)^{p+1} \; \;\text{(by Young's ineq.)}
\end{align*}
for all $t \in (0, T_{\max})$. The lemma then follows with the fact  $\int_\Omega w \leq\frac{\max\{\nu,\lambda\}}{\mu}\int_\Omega (u+v)$ and taking $C=\tilde C_1(\frac{\max\{\nu,\lambda\}}{\mu})^{p+1}$ and $\tilde \varepsilon =\frac{2p}{(p+1)^2}(\frac{p+1}{\max\{\nu,\lambda\}})^{p+1}(\frac{\mu}{p})^p \varepsilon.$
\end{proof}}

\section{Proofs of Theorems \ref{new-main-thm1}, and \ref{new-main-thm4}}

In this section, {  we investigate lower bound of the combined mass,  and $L^{-q}$-boundedness    of classical solutions of \eqref{main-eq}, } and  prove Theorems \ref{new-main-thm1},
and \ref{new-main-thm4}.

Throughout this section, for given $u_0,v_0$ satisfying   \eqref{initial-cond-eq}, we put
$$
(u(t,x),v(t,x),w(t,x)):=(u(t,x;u_0,v_0),v(t,x;u_0,v_0),w(t,x;u_0,v_0)),
$$
and if no confusion occurs, we may drop $(t,x)$ in $u(t,x)$ (resp. $v(t,x)$, $w(t,x)$).

\subsection{Lower bound of the combined mass on bounded intervals and proof of Theorem \ref{new-main-thm1}}

In this subsection, we study the lower  bound of the combined mass on bounded intervals and proof of Theorem \ref{new-main-thm1}.

We   first present a lemma on $\int_\Omega\ln(u+v)dx$. {Note that, by \eqref{positive-solu-eq},
$u(t,x;u_0,v_0)+v(t,x;u_0,v_0)>0$ for all $t\in (0,T_{\max}(u_0,v_0))$ and $x\in\bar\Omega$,. Hence $\int_\Omega\ln (u(t,x;u_0,v_0)+v(t,x;u_0,v_0))dx$ is well defined for $t\in (0,T_{\max}(u_0,v_0))$.}

\begin{lemma}
\label{mass-persistence-lm1} For any $u_0(x)$ and $v_0(x)$ satisfying  \eqref{initial-cond-eq},
there exists $K=K(u_0,v_0)>0$ such that
\begin{equation*}
    \frac{d}{dt} \int_{\Omega} \ln\big(u(t,x;u_0,v_0)+v(t,x;u_0,v_0)\big)dx \ge -K \quad \text{\rm for all}\,\, t\in (0,T_{\max}(u_0,v_0)).
\end{equation*}
\end{lemma}

\begin{proof}
First, multiplying \eqref{new-u+v-eq1-0}  by $\frac{1}{u+v}$ and then  integrating on $\Omega$ yields that
\begin{align*}
    \frac{d}{d t} \int_{\Omega} \ln{(u+v)}\ge& \int_{\Omega} \frac{|\nabla (u+v)|^2}{(u+v)^2} - \underbrace{ \int_{\Omega} \frac{\chi_1 u + \chi_2 v}{(u+v)^2} \nabla (u+v) \cdot \frac{\nabla w}{w}}_{I} \\
    &+ \int_{\Omega} \frac{a_1 u+a_2 v}{u+v} - \int_{\Omega} \frac{b_1u^2+b_2 v^2}{u+v} - \int_{\Omega} \frac{(c_1+c_2)u v }{u+v}
\end{align*}
for all $t \in (0,T_{\max}(u_0,v_0)).$
By Young's inequality and Lemma \ref{pre-lm-6}, we have
\begin{align*}
    I=\int_{\Omega} \frac{\chi_1 u + \chi_2 v}{(u+v)^2} \nabla (u+v) \cdot \frac{\nabla w}{w} & \leq \int_{\Omega} \frac{|\nabla (u+v)|^2}{(u+v)^2}+ \frac{\max\{\chi_1^2,\chi_2^2\}}{4} \int_{\Omega}\frac{|\nabla w|^2}{w^2}\\
    & \leq \int_{\Omega} \frac{|\nabla (u+v)|^2}{(u+v)^2}+ \frac{\mu |\Omega| \max\{\chi_1^2,\chi_2^2\}}{4}.
\end{align*}
Then by Lemma \ref{pre-lm-5} we have
\begin{equation*}
    \frac{d}{dt} \int_{\Omega} \ln{(u+v)}\ge - \frac{\mu |\Omega|\max\{\chi_1^2,\chi_2^2\}}{4} +a_{\min}|\Omega|-\left(b_{\max}+c_{\max}\right) (m_1^*(0,u_0,v_0)+m_2^*(0,u_0,v_0))
\end{equation*}
for all $t\in (0,T_{\max})$.
The lemma is thus proved.
\end{proof}

We now prove Theorem \ref{new-main-thm1}.

\begin{proof}[Proof of Theorem \ref{new-main-thm1}]
 Fix a $\tau\in (0,T_{\max}(u_0,v_0))$. It is clear that
$$\inf_{0\le t\le \tau}\int_\Omega \big(u(t,x;u_0,v_0)+v(t,x;u_0,v_0)) dx>0.
$$
It then suffices to prove that  for any $T>0$,  there exist $C=C(T)>0$  such that
\begin{equation}
\label{new-new-eq1}
    \int_{\Omega} \big(u(t,x;u_0,v_0)+v(t,x;u_0,v_0)\big )dx \ge C(T) \quad \text{\rm for all}\,\, t\in (\tau, \hat T),
\end{equation}
where $\hat T=\min\{T,T_{\max}(u_0,v_0)\}$.

Note that
 $L:=\int_{\Omega} \ln{\big(u(\tau,x;u_0,v_0)+v(\tau,x;u_0,v_0)\big)}dx$ is finite. By Lemma \ref{mass-persistence-lm1},
 there exists $K=K(u_0,v_0)>0$ such that
$$ \frac{d}{dt} \int_{\Omega} \ln{\big(u(t,x;u_0,v_0)+v(t,x;u_0,v_0)\big)dx}\ge -K
$$
 for all $t \in (0,T_{\max}(u_0,v_0)) $.  We thus have that
\begin{align*}
    \int_{\Omega} \ln{\big(u(t,x;u_0,v_0)+v(t,x;u_0,v_0)\big)}dx &\ge \int_{\Omega} \ln{(u(\tau,x;u_0,v_0)+v(\tau,x;u_0,v_0))}dx-K\cdot(t-\tau)\\
    & \ge L-K\cdot(\hat T-\tau)=C(K,L,\tau
    ) \quad \text{for all}\,\, t\in (\tau,\hat T).
\end{align*}
Therefore Jensen's inequality asserts that
\begin{align*}
    \int_{\Omega} \ln{\big(u(t,x;u_0,v_0)+v(t,x;u_0,v_0)\big)}dx&=|\Omega| \cdot \int_{\Omega} \ln{\big(u(t,x;u_0,v_0)+v(t,x;u_0,v_0)\big)}\frac{dx}{|\Omega|} \\
&\leq |\Omega| \cdot \ln{\Big( \int_{\Omega} \big(u(t,x;u_0,v_0)+v(t,x;u_0,v_0)\big) \frac{dx}{|\Omega|}  \Big)}
\end{align*}
for all $ t\in (0,  T_{\max}(u_0,v_0))$,
which implies
\begin{align*}
    \int_{\Omega}\big( u(t,x;u_0,v_0)+v(t,x;u_0,v_0)\big)dx& \ge |\Omega| \cdot \exp{\Big( \frac{1}{|\Omega|} \cdot \int_{\Omega} \ln{\big(u(t,x;u_0,v_0)+v(t,x;u_0,v_0)\big)} dx   \Big)} \\
&\ge |\Omega| \cdot e^{\frac{1}{|\Omega|}\cdot C(K,L,\tau)} \quad \text{for all}\,\, t\in (\tau, \hat T),
\end{align*}
which implies \eqref{new-new-eq1}.
The theorem is  thus proved.
\end{proof}

\subsection{$L^{-q}$-boundedness and proof of  Theorem \ref{new-main-thm4}}

{ In this subsection, we study the boundedness of  $\int_\Omega(u+v)^{-q}$
 and prove Theorem \ref{new-main-thm4}.}

We first prove Theorem \ref{new-main-thm4}(1). To this end.
we first prove a lemma.

\begin{lemma}
\label{mass-persistence-thm} Let  $\mu>0$, $\chi_1>0$ and $\chi_2>0$ be given.
\begin{itemize}
\item[(1)]
For any $B>0$ and $0<\beta\not = { \chi_2-B}$,
\begin{align}
\label{new-LP-eq7-3-0}
\frac{1}{q}\frac{d}{dt}\int_\Omega (u+v)^{-q}
\le&  \Big(f(\mu,\chi_1,\chi_2,\beta,B)- a_{\min}\Big) \int_\Omega (u+v)^{-q} \nonumber\\
&+(b_{\max}+c_{\max}) \int_\Omega (u+v)^{-q+1},\quad { \forall\, t\in (0,T_{\max}(u_0,v_0)),}
\end{align}
where $f(\mu,\chi_1,\chi_2,\beta,B)$ is as in \eqref{function-f-eq} and
\begin{equation}
\label{new-p-eq0}
q=\frac{4B\beta}{B( \chi_2-B-\beta)^2+(\chi_1-\chi_2)^2\beta}.
\end{equation}

\item[(2)] For any $B>0$ and $0<\beta\not ={ \chi_1 -B}$,
\begin{align}
\label{new-LP-eq7-3-1}
\frac{1}{q}\frac{d}{dt}\int_\Omega (u+v)^{-q}
\le&  \Big(f(\mu,\chi_2,\chi_1,\beta,B)- a_{\min}\Big) \int_\Omega (u+v)^{-q} \nonumber\\
&+(b_{\max} +c_{\max})\int_\Omega (u+v)^{-q+1},\quad { \forall\, t\in (0,T_{\max}(u_0,v_0))},
\end{align}
where
\begin{equation}
\label{new-p-eq00}
q=\frac{4B\beta}{B({ \chi_1-B}-\beta)^2+(\chi_1-\chi_2)^2\beta}.
\end{equation}

\end{itemize}
\end{lemma}

\begin{remark}
\label{elliptic-rk2}
{Lemma \ref{mass-persistence-thm}    plays  a crucial role in the proof of the global existence and  boundedness of the solutions of \eqref{main-eq}.}
The proof of Lemma \ref{mass-persistence-thm} strongly relies
on the fact that the third equation in \eqref{main-eq} is elliptic. Therefore  new techniques need to be developed  to investigate  the global existence and boundedness of solutions of \eqref{full-parabolic-eq}.
\end{remark}

\begin{proof} [Proof of Lemma \ref{mass-persistence-thm}]
{ In this proof, we always assume that $t\in (0,T_{\max}(u_0,v_0))$.}

(1)  First,   by \eqref{new-LP-eq0-0}, we have
\begin{align}
\label{new-LP-eq0-01}
\frac{1}{q}\frac{d}{dt}\int_\Omega (u+v)^{-q}
\le &- (q+1)\int_\Omega (u+v)^{-q-2}|\nabla (u+v)|^2\nonumber\\
&+(q+1)\chi_2 \int_\Omega
\frac{(u+v)^{-q-1}}{w}\nabla (u+v)\cdot\nabla w\nonumber\\
&+   \frac{(\chi_1-\chi_2)^2 q(q+1)}{4B}\int_\Omega
(u+v)^{-q-2}|\nabla (u+v)|^2 +\frac{B(q+1)}{q}\int_\Omega (u+v)^{-q}\frac{|\nabla w|^2}{w^2}\nonumber\\
& - a_{\min} \int_\Omega (u+v)^{-q}+(b_{\max}+c_{\max}) \int_\Omega (u+v)^{-q+1}
\end{align}
for any $q>0$ and $B>0$.

Next, multiplying the third equation in \eqref{main-eq} by $\frac{(u+v)^{-q}}{w}$ and then integrating over $\Omega$ with respect to $x$, we obtain that
\begin{align*}
    0&= \int_{\Omega} \frac{(u+v)^{-q}}{w} \cdot \big( \Delta w-\mu w +\nu u+ \lambda v\big)\\
    &= - \int_{\Omega}\frac{(-q) (u+v)^{-q-1}w \nabla (u+v) -(u+v)^{-q} \nabla w}{w^2} \cdot \nabla w\\
    &\quad - \mu \int_{\Omega} (u+v)^{-q} + \nu   \int_{\Omega}\frac{u (u+v)^{-q}}{w} + \lambda  \int_{\Omega}\frac{v (u+v)^{-q}}{w}
\end{align*}
for all ${t\in (0, T_{\max})}.$ Thus we have,
\begin{equation}
\label{new-u+v-eq2}
     q \int_{\Omega} \frac{(u+v)^{-q-1}}{w} \nabla (u+v) \cdot \nabla w + \int_{\Omega} (u+v)^{-q} \frac{|\nabla w|^2}{w^2} \le  \mu \int_{\Omega} (u+v)^{-q}
\end{equation}
for all $t\in (0,T_{\max}).$ It then follows that
$$
\frac{B}{q}\int_\Omega (u+v)^{-q}\frac{|\nabla w|^2}{w^2}\le \frac{B\mu}{q}\int_\Omega (u+v)^{-q}-B\int_\Omega \frac{(u+v)^{-q-1}}{w}\nabla(u+v)\cdot\nabla w.
$$
This together with \eqref{new-LP-eq0-01}  implies that
\begin{align}
\label{new-LP-eq0-02}
\frac{1}{q}\frac{d}{dt}\int_\Omega(u+v)^{-q} \le &- (q+1)\int_\Omega (u+v)^{-q-2}|\nabla (u+v)|^2\nonumber\\
&+(q+1)(\chi_2-B) \int_\Omega
\frac{(u+v)^{-q-1}}{w}\nabla (u+v)\cdot\nabla w\nonumber\\
&+   \frac{(\chi_1-\chi_2)^2 q(q+1)}{4B}\int_\Omega
(u+v)^{-q-2}|\nabla (u+v)|^2 \nonumber\\
&+ \frac{B(q+1)\mu}{q}\int_\Omega (u+v)^{-q}\nonumber\\
& - a_{\min} \int_\Omega (u+v)^{-q}+(b_{\max}+c_{\max}) \int_\Omega (u+v)^{-q+1}
\end{align}
for any $q>0$ and $B>0$.

Now, let ${ \tilde \chi=\chi_2-B}$.
Note that, when $B>0$ and $q>0$ are such that $D:=\frac{(\chi_1-\chi_2)^2}{4B}q<1$,  for any $\beta>0$,
by Young's inequality and \eqref{new-u+v-eq2},  we have
\begin{align*}
\tilde \chi\int_\Omega \frac{(u+v)^{-q-1}}{w}\nabla (u+v)\cdot\nabla w&=(\tilde\chi-\beta)\int_\Omega \frac{(u+v)^{-q-1}}{w}\nabla (u+v)\cdot\nabla w\nonumber\\
&\quad +\beta \int_\Omega \frac{(u+v)^{-q-1}}{w}\nabla (u+v)\cdot\nabla w\nonumber\\
&\le (1-D)\int_\Omega (u+v)^{-q-2}|\nabla(u+v)|^2 +\frac{|\tilde \chi-\beta|^2}{4(1-D)}\int _\Omega (u+v)^{-q}\frac{|\nabla w|^2}{w^2}\nonumber\\
&\quad +\frac{\beta \mu}{q}\int_\Omega(u+v)^{-q}-\frac{\beta}{q}\int_\Omega (u+v)^{-q}\frac{|\nabla w|^2}{w^2}.
\end{align*}
This together with  \eqref{new-LP-eq0-02} implies that
\begin{align}
\label{new-LP-eq0-03}
\frac{1}{q}\frac{d}{dt}\int_\Omega(u+v)^{-q}\le&  (q+1)\Big(\frac{B|\tilde \chi-\beta|^2}{4B-(\chi_1-\chi_2)^2q}-\frac{\beta}{q}\Big)\int_\Omega (u+v)^{-q}\frac{|\nabla w|^2}{w^2}\nonumber\\
&\quad  +(q+1)\Big(\frac{B\mu}{q}+\frac{\beta\mu}{q}\Big)\int_\Omega (u+v)^{-q}\nonumber\\
&  - a_{\min} \int_\Omega (u+v)^{-q}+(b_{\max}+c_{\max}) \int_\Omega (u+v)^{-q+1}
\end{align}
for any $\beta>0$ and $B>0$, $q>0$ with $\frac{(\chi_1-\chi_2)^2}{4B}{ q}<1$.

\smallskip

Finally,   for any  $B>0$ and $\beta>0$, $\beta\not =\tilde \chi$, let $q$ be as in \eqref{new-p-eq0}.
Then
\begin{equation*}
\frac{(\chi_1-\chi_2)^2}{4B}q<1,
\end{equation*}
 and
\begin{equation*}
\frac{B|\tilde \chi-\beta|^2}{4B-(\chi_1-\chi_2)^2 q}-\frac{\beta}{q}= 0.
\end{equation*}
Then by \eqref{new-LP-eq0-03},  we have
\begin{align*}
\frac{1}{q}\frac{d}{dt}\int_\Omega (u+v)^{-q}
\le&  \Big(({B\mu+\beta \mu})\Big(1+\frac{1}{q}\Big)- a_{\min}\Big) \int_\Omega (u+v)^{-q}\nonumber\\
& +(b_{\max}+c_{\max}) \int_\Omega (u+v)^{-q+1}.
\end{align*}
Note that
$$
(B\mu+\beta \mu)\Big(1+\frac{1}{q}\Big)=f(\mu,\chi_1,\chi_2,\beta, B).
$$
This implies that \eqref{new-LP-eq7-3-0} holds.

(2) It can be proved by the similar arguments as in (1).
\end{proof}

{
\begin{remark}
\label{q-rk1}
\begin{itemize}

\item[(1)]  {If  $(\chi_1-\chi_2)^2\le { 4\chi_2}$, letting  $B=\chi_2$ and $q=1$},
 by  \eqref{new-LP-eq0-02}, we have
\begin{align*}
\frac{1}{q}\frac{d}{dt}\int_\Omega (u+v)^{-1}
\le&  \Big(2\mu \chi_2- a_{\min}\Big) \int_\Omega (u+v)^{-1} \nonumber\\
&+(b_{\max}+c_{\max}) |\Omega|,\quad { \forall\, t\in (0,T_{\max}(u_0,v_0)).}
\end{align*}
Similarly, { if $(\chi_1-\chi_2)^2\le 4\chi_1$, }
 letting $B=\chi_1$ and $q=1$, we have
\begin{align*}
\frac{1}{q}\frac{d}{dt}\int_\Omega (u+v)^{-1}
\le&  \Big(2\mu \chi_1- a_{\min}\Big) \int_\Omega (u+v)^{-1} \nonumber\\
&+(b_{\max}+c_{\max}) |\Omega|,\quad { \forall\, t\in (0,T_{\max}(u_0,v_0)).}
\end{align*}

\item[(2)]
Assume that $(\chi_1-\chi_2)^2>{ \max \{4\chi_1,4\chi_2\}}$.  Let  $B=\chi_2$ and $q=\frac{4\chi_2}{(\chi_1-\chi_2)^2}$.  By
\eqref{new-LP-eq0-02} again, we have
\begin{align*}
\frac{1}{q}\frac{d}{dt}\int_\Omega (u+v)^{-q}
\le&  \Big(\mu \chi_2+\frac{\mu (\chi_1-\chi_2)^2}{4}- a_{\min}\Big) \int_\Omega (u+v)^{-q} \nonumber\\
&+(b_{\max}+c_{\max}) \int_\Omega u^{-q+1},\quad { \forall\, t\in (0,T_{\max}(u_0,v_0)).}
\end{align*}
Similarly,    let $B=\chi_1$ and $q=\frac{4\chi_1}{(\chi_1-\chi_2)^2}$.  We have
\begin{align*}
\frac{1}{q}\frac{d}{dt}\int_\Omega (u+v)^{-q}
\le&  \Big(\mu \chi_1+\frac{\mu (\chi_1-\chi_2)^2}{4}- a_{\min}\Big) \int_\Omega (u+v)^{-q} \nonumber\\
&+(b_{\max}+c_{\max})\int_\Omega u^{-q+1},\quad {\forall\, t\in (0,T_{\max}(u_0,v_0)).}
\end{align*}

\end{itemize}
\end{remark}
}

\begin{proof}[Proof of Theorem \ref{new-main-thm4}(1)]
{In this proof, we always assume that $t\in (0,T_{\max}(u_0,v_0))$ unless specified otherwise.}

  Without loss of generality, we assume that
$$
\chi^*(\mu,\chi_1,\chi_2)=\min\{\chi_1^*(\mu,\chi_1,\chi_2),\chi_2^*(\mu,\chi_1,\chi_2)\}=\chi_1^*(\mu,\chi_1,\chi_2).
$$
Let ${\epsilon_0}=a_{\min}-\chi^*(\mu,\chi_1,\chi_2)$.
Then by the definition of $\chi_1^*(\mu,\chi_1,\chi_2)$, there are $B>0$ and $0<\beta\not= { \chi_2-B}$ such that
$$
a_{\min}-f(\mu,\chi_1,\chi_2,\beta, B)>\frac{3\epsilon_0}{4}.
$$
By  \eqref{new-LP-eq7-3-0}, we then have
\begin{align}
\label{new-LP-eq7-00}
\frac{1}{q}\frac{d}{dt}\int_\Omega (u+v)^{-q}
\le -\frac{3{ \epsilon_0}}{4} \int_\Omega (u+v)^{-q} +(b_{\max}+c_{\max}) \int_\Omega (u+v)^{-q+1}
\end{align}
for all $t>0$,
where  $q$ is as in  \eqref{new-p-eq0}.

Note that if $q\ge 1$, by  Young's inequality,
\begin{align*}
    (b_{\max} +c_{\max})\int_\Omega (u+v)^{-q+1} \leq& \frac{{ \epsilon_0}}{4} \int_\Omega (u+v)^{-q}+C_{\epsilon_0},
\end{align*}
{where $C_{\epsilon_0}=\frac{(b_{\max}+c_{\max})^q}{q}\big(\frac{4}{\epsilon_0}\big)^q\big(\frac{q-1}{q}\big)^{q-1}|\Omega|$ if $q>1$ and $C_{\epsilon_0}=(b_{\max}+c_{\max})|\Omega|$ if $q=1$.}
If $q< 1$, then by the H\"older inequality
\begin{align}
\label{b-c-max-2-0}
    (b_{\max}+c_{\max}) \int_\Omega (u+v)^{-q+1} \leq&  (b_{\max} +c_{\max})|\Omega|^q \cdot \left(\int_{\Omega} u+v\right)^{1-q}\nonumber\\
 \leq & (b_{\max}+c_{\max}) |\Omega|^q  (m^*(\tau,u_0,v_0))^{1-q}
\end{align}
for any { $ 0< \tau<t<T_{\max}(u_0,v_0)$, where $m^*(\tau,u_0,v_0)$ is as in \eqref{new-u+v-upper-bound-eq}.}
By \eqref{new-LP-eq7-00}-\eqref{b-c-max-2-0},   there is $\tilde C_{\epsilon_0, \tau,u_0,v_0}>0$ such that
\begin{equation*}
\frac{1}{q}\frac{d}{dt}\int_\Omega (u+v)^{-q}
\le  -\frac{\epsilon_0}{2}  \int_\Omega (u+v)^{-q}+{  C_{\epsilon_0, \tau,u_0,v_0}}\quad \forall\, { 0<\tau<t<T_{\max}(u_0,v_0)},
\end{equation*}
{ where
\begin{equation}
    \label{c-eps-tau-1}
    C_{\epsilon_0, \tau,u_0,v_0}=\begin{cases}
        \frac{(b_{\max}+c_{\max})^q}{q}\big(\frac{4}{\epsilon_0}\big)^q\big(\frac{q-1}{q}\big)^{q-1}|\Omega| \quad &\text{if} \;\; q > 1, \cr
        (b_{\max}+c_{\max})|\Omega|\quad &\text{if} \;\; q = 1, \cr
        (b_{\max}+c_{\max})|\Omega|^q (m^*(\tau,u_0,v_0))^{1-q}  \quad &\text{if} \;\; q <1.
    \end{cases}
\end{equation}}
This implies that
\begin{align}
\label{new-LP-eq7-4}
&\int_\Omega (u(x,t;u_0,v_0)+v(x,t;u_0,v_0))^{-q}dx\nonumber\\
& \le e^{-\frac{\epsilon_0 q}{2}(t-\tau)} \int_\Omega(u(x,\tau;u_0,v_0)+v(x,\tau;u_0,v_0))^{-q}dx+ { 2   C_{\epsilon_0,\tau,u_0,v_0} \epsilon_0^{-1}}
\end{align}
for any $\tau>0$ and  ${\tau<t<T_{\max}(u_0,v_0)}$.  { Theorem \ref{new-main-thm3}(1)  thus follows in the case that $a_{\min}>\chi^*(\mu,\chi_1,\chi_2)$.}
\end{proof}

 Next, we prove Theorem \ref{new-main-thm4}(2).

\begin{proof}[Proof of Theorem \ref{new-main-thm4}(2)]
{ In this proof, we always assume that $t\in (0,T_{\max}(u_0,v_0))$ unless specified otherwise.}

  Assume that $a_{\min}>\chi^{**}(\mu,\chi_1,\chi_2)$.  Let $q=q_{\chi_1,\chi_2}$. By Remark \ref{q-rk1}, we have
\begin{equation}
\label{aux-new-LP-eq1}
\frac{1}{q}\frac{d}{dt}\int_\Omega(u+v)^{-q} \le -\big(a_{\min}-\chi^{**}(\mu,\chi_1,\chi_2)\big)\int_\Omega(u+v)^{-q}+(b_{\max}+c_{\max})\int_\Omega (u+v)^{-q+1}.
\end{equation}
Note that $q_{\chi_1,\chi_2}\le 1$. By the arguments of \eqref{new-LP-eq7-4}, we have
\begin{align*}
&\int_\Omega (u(x,t;u_0,v_0)+v(x,t;u_0,v_0))^{-q}dx\nonumber\\
& \le e^{-(a_{\min}-\chi^{**}(\mu,\chi_1,\chi_2)) { q} (t-\tau)} \int_\Omega(u(x,\tau;u_0,v_0)+v(x,\tau;u_0,v_0))^{-q}dx+ \tilde   C_{\tau,u_0,v_0},
\end{align*}
where
\begin{equation*}
\tilde C_{\tau,u_0,v_0}=
\begin{cases}
\frac{(b_{\max}+c_{\max})|\Omega|}{a_{\min}-\chi^{**}(\mu,\chi_1,\chi_2)}\,\,\quad  &{\rm if}\;\; q=1,\cr\cr
\frac{(b_{\max}+c_{\max})|\Omega|^{q}\big(m^*(\tau,u_0,v_0)\big)^{1-q}}{a_{\min}-\chi^{**}(\mu,\chi_1,\chi_2)}  \,\,\quad  &{\rm if}\;\; q<1.
\end{cases}
\end{equation*}
Theorem \ref{new-main-thm3}(2)  follows in the case that $a_{\min}>\chi^{**}(\mu,\chi_1,\chi_2)$.
\end{proof}

\begin{remark}
\label{q-rk2}  Assume that
$a_{\min}>\chi^{**}(\mu,\chi_1,\chi_2)$.
By \eqref{aux-new-LP-eq1},  we also have
\begin{align*}
&\int_\Omega( u+v)^{-q_{\chi_1,\chi_2}}(t,x)dx\nonumber\\
&\le
\begin{cases}
\max\Big\{\int_\Omega( u+v)^{-q_{\chi_1,\chi_2}}(\tau,x)dx,\frac{(b_{\max}+c_{\max})|\Omega|}{a_{\min}-\chi^{**}(\mu,\chi_1,\chi_2)}\Big\}\cr
\qquad \quad\qquad\qquad {\rm if}\; (\chi_1-\chi_2)^2\le { \max\{4\chi_1,4\chi_2\}}\cr\cr
\max\Big\{\int_\Omega (u+v)^{-q_{\chi_1,\chi_2}}(\tau,x)dx, \frac{(b_{\max}+c_{\max})|\Omega|^{q_{\chi_1,\chi_2}}\big(m^*(\tau,u_0,v_0)\big)^{1-q_{\chi_1,\chi_2}}}{a_{\min}-\chi^{**}(\mu,\chi_1,\chi_2)}  \Big\}\cr
\qquad \quad\qquad\qquad  {\rm if}\; (\chi_1-\chi_2)^2>{ \max \{4\chi_1,4\chi_2\}}
\end{cases}
\end{align*}
for all $0<\tau\le t<T_{\max}(u_0,v_0)$.
\end{remark}

\section{$L^p$- and $C^\theta$-boundedness and proof of Theorem \ref{new-main-thm2}}

In this section, we study the  $L^p$- and $C^\theta$-boundedness of classical solutions of \eqref{main-eq} and prove Theorem \ref{new-main-thm2}.

\subsection{Proof of Theorem \ref{new-main-thm2}(1)}

In this subsection, we prove Theorem \ref{new-main-thm2}(1).
To this end, we first make some observations.
In the following, again  we always assume that
$t\in (0, T_{\max}(u_0,v_0))$ unless specified otherwise.

Fix $p>2$. First, observe that
\begin{align}
\label{new-new-u+v-eq1}
\frac{1}{p}\frac{d}{dt}\int_\Omega (u+v)^{p}=& -(p-1)\int_\Omega (u+v)^{p-2}|\nabla (u+v)|^2\nonumber\\
& +(p-1)\int_\Omega (u+v)^{p-2}\frac{\chi_1 u+\chi_2 v}{w} \nabla(u+v)\cdot\nabla w\nonumber\\
& +\int_\Omega (u+v)^{p-1}(a_1u+a_2 v)-\int_\Omega (u+v)^{p-1}(b_1 u^2+b_2 v^2)-\int_\Omega (u+v)^{p-1}(c_1+c_2) uv\nonumber\\
\le &  -(p-1)\int_\Omega (u+v)^{p-2}|\nabla (u+v)|^2\nonumber\\
& +(p-1)\int_\Omega (u+v)^{p-2}\frac{\chi_1 u+\chi_2 v}{w} \nabla(u+v)\cdot\nabla w\nonumber\\
& +a_{\max} \int_\Omega (u+v)^{p} - \min\{b_{\min},c_{\min}\} \int_\Omega (u+v)^{p+1}.
\end{align}

To prove Theorem \ref{new-main-thm2}(1), it is then essential to provide proper estimates for the integral   $\int_\Omega (u+v)^{p-2}\frac{\chi_1 u+\chi_2 v}{w} \nabla(u+v)\cdot\nabla w$.

Next, observe  that
\begin{align*}
     \int_\Omega \frac{(u+v)^{p-2}(\chi_1 u+\chi_2 v)}{w} \nabla(u+v)\cdot\nabla w & =\chi_2 \int_\Omega \frac{(u+v)^{p-1}}{w}  \nabla(u+v)\cdot\nabla w\nonumber\\
     &\quad + (\chi_1-\chi_2) \int_\Omega \frac{(u+v)^{p-2}}{w} u \nabla(u+v)\cdot\nabla w.
\end{align*}
{ Since $u^2 \leq (u+v)^2,$ by  Young's inequality,
we have that
\begin{align*}
  &  \big|(\chi_1-\chi_2)\int_\Omega \frac{(u+v)^{p-2}}{w} u \nabla(u+v)\cdot\nabla w\big|\nonumber\\
 &\leq \frac{1}{2} \int_\Omega (u+v)^{p-2} |\nabla (u+v)|^2 + \frac{(\chi_1-\chi_2)^2}{2} \int_\Omega (u+v)^{p} \frac{|\nabla w|^2}{w^2}
\end{align*}
and
$$
|\chi_2\int_\Omega \frac{(u+v)^{p-1}}{w}\nabla (u+v)\cdot \nabla w|\le \frac{1}{2}\int_\Omega (u+v)^{p-2}|\nabla(u+v)|^2+\frac{\chi_2^2}{2}\int_\Omega (u+v)^p\frac{|\nabla w|^2}{w^2}.
$$
Therefore,
\begin{align}
\label{new-new-u+v-eq2}
  &\int_\Omega \frac{(u+v)^{p-2}(\chi_1 u+\chi_2 v)}{w} \nabla(u+v)\cdot\nabla w\nonumber\\
&\le \int_\Omega (u+v)^{p-2}|\nabla(u+v)|^2+\frac{\chi_2^2+(\chi_1-\chi_2)^2}{2}\int_\Omega (u+v)^p\frac{|\nabla w|^2}{w^2}\nonumber\\
&\le \int_\Omega (u+v)^{p-2}|\nabla (u+v)|^2+\frac{\chi_2^2+(\chi_1-\chi_2)^2}{2\displaystyle\inf_{x\in\Omega}w(t,x)}\int_\Omega (u+v)^p\frac{|\nabla w|^2}{w}\nonumber\\
&\le \int_\Omega (u+v)^{p-2}|\nabla (u+v)|^2+
\frac{\chi_2^2+(\chi_1-\chi_2)^2}{2\displaystyle\inf_{x\in\Omega}w(t,x)}\Big(\int_\Omega (u+v)^{p+1}\Big)^{\frac{p}{p+1}}\Big(\int_\Omega \frac{|\nabla w|^{2p+2}}{w^{p+1}}\Big)^{\frac{1}{p+1}}.
\end{align}
{Similarly,} we have
\begin{align}
\label{new-new-u+v-eq3}
  &\int_\Omega \frac{(u+v)^{p-2}(\chi_1 u+\chi_2 v)}{w} \nabla(u+v)\cdot\nabla w\nonumber\\
&\le \int_\Omega (u+v)^{p-2}|\nabla (u+v)|^2+
\frac{ \chi_1^2+(\chi_1-\chi_2)^2}{2\displaystyle\inf_{x\in\Omega}w(t,x)}\Big(\int_\Omega (u+v)^{p+1}\Big)^{\frac{p}{p+1}}\Big(\int_\Omega \frac{|\nabla w|^{2p+2}}{w^{p+1}}\Big)^{\frac{1}{p+1}}.
\end{align}
}

{ To provide proper estimates for  $\int_\Omega (u+v)^{p-2}\frac{\chi_1 u+\chi_2 v}{w} \nabla(u+v)\cdot\nabla w$,
it is then essential to provide  proper estimates  for $\int_\Omega \frac{|\nabla w|^{2p+2}}{w^{p+1}}$.
We have the following  proposition on estimates for   $\int_\Omega \frac{|\nabla w|^{2p}}{w^{k}}$. }

\begin{proposition}
\label{new-main-prop2}
Let $p \ge  3$ and $p-\sqrt{2p-3} < k <  p+\sqrt{2p-3}.$  There are $M(p,k)>0$ and $\tilde M(p,k)>0$ such that
 \begin{equation*}
   \int_{\Omega} \frac{|\nabla w|^{2p}}{w^{k}} \leq M(p,k)\int_{\Omega} \frac{(\nu u+\lambda v)^p}{w^{k-p}}  + \tilde M(p,k) \int_{\Omega} w^{2p-k}\quad \forall\,\, t \in (0, T_{\rm max}(u_0,v_0)).
\end{equation*}
\end{proposition}

\begin{remark}
\label{elliptic-rk1}
We remark that the estimate in Proposition \ref{new-main-prop2} plays a crucial role in the proof of the boundedness of the solutions of \eqref{main-eq}.
The proof of Proposition \ref{new-main-prop2} strongly relies
on the fact that the third equation in \eqref{main-eq} is elliptic. It therefore requires new techniques to investigate  the global existence and boundedness of solutions of \eqref{full-parabolic-eq}.
\end{remark}

In the rest of this subsection, we  prove Theorem \ref{new-main-thm2}(1)  by applying Proposition \ref{new-main-prop2}.
 Proposition \ref{new-main-prop2} can be proved by the
similar arguments as those in  { \cite[Proposition 1.3]{HKWS}}. For the reader's convenience, we will provide the outline of the proof of   Proposition  \ref{new-main-prop2} in  the appendix.

\begin{proof} [Proof of Theorem \ref{new-main-thm2}(1)]
 { First,  for any $p>\max\{2,N\}$,
by   \eqref{new-new-u+v-eq1},  \eqref{new-new-u+v-eq2},  and    \eqref{new-new-u+v-eq3},  we have
\begin{align}
\label{aux-u+v-eq1}
\frac{1}{p}\frac{d}{dt}\int_\Omega (u+v)^{p}&  \le \frac{(p-1)\big(\chi_i^2+(\chi_1-\chi_2)^2\big)}{2\displaystyle\inf_{x\in\Omega}w(t,x)} \Big(\int_\Omega (u+v)^{p+1}\Big)^{\frac{p}{p+1}}\Big(\int_\Omega\frac{|\nabla w|^{2p+2}}{w^{p+1}}\Big)^{\frac{1}{p+1}}\nonumber\\
&\quad +a_{\max} \int_\Omega (u+v)^{p} - \min\{b_{\min},c_{\min}\} \int_\Omega (u+v)^{p+1}
\end{align}
for all $t\in (0,T_{\max}(u_0,v_0))$ and $i=1,2$.
Let
$$
C^*(p)=M(p+1,p+1),\quad C^{**}(p)=\tilde M(p+1,p+1),
$$
and
$$
\tilde \chi= \min\Big\{\chi_1^2+(\chi_1-\chi_2)^2,\chi_2^2+(\chi_1-\chi_2)^2\Big\}.
$$
By \eqref{aux-u+v-eq1} and Proposition \ref{new-main-prop2}, we have
\begin{align}
\label{aux-u+v-eq2}
\frac{1}{p}\frac{d}{dt}\int_\Omega (u+v)^{p}
&  \le  \frac{\tilde\chi { (p-1)}}{2\displaystyle\inf_{x\in\Omega}w(t,x)} \Big(\int_\Omega (u+v)^{p+1}\Big)^{\frac{p}{p+1}}\Big(C^*(p)\int_\Omega{ ( \nu u+\lambda v)}^{p+1}+C^{**}(p) \int_\Omega w^{p+1}\Big)^{\frac{1}{p+1}}\nonumber\\
&\quad +a_{\max} \int_\Omega (u+v)^{p} - \min\{b_{\min},c_{\min}\} \int_\Omega (u+v)^{p+1}\nonumber\\
&\le \frac{\tilde\chi  { (p-1)}\big(C^*(p)\big)^{\frac{1}{p+1}}{ \max\{\nu,\lambda\}}}{2\displaystyle\inf_{x\in\Omega}w(t,x)}\int_\Omega (u+v)^{p+1}\nonumber\\
&
+\frac{\tilde \chi { (p-1)} \big(C^{**}(p)\big)^{\frac{1}{p+1}}}{2\displaystyle\inf_{x\in\Omega}w(t,x)}\Big(\int_\Omega (u+v)^{p+1}\Big)^{\frac{p}{p+1}}\Big(\int_\Omega  w^{p+1}\Big)^{\frac{1}{p+1}}\nonumber\\
&\quad +a_{\max} \int_\Omega (u+v)^{p} - \min\{b_{\min},c_{\min}\} \int_\Omega (u+v)^{p+1}.
\end{align}

Next,
by {\bf (H2)} and Remark \ref{q-rk2},
$$
\int_\Omega (u+v)^{-q^*}\le \frac{(b_{\max}+c_{\max})|\Omega|^{q^*}\big(m^*\big)^{1-q^*}}{\big(a_{\min}-\chi^{**}(\mu,\chi_1,\chi_2)\big){{ C_{\chi_1,\chi_2}^{ q^*}}}} ,\,\forall\,\, \tau_0\le t<T_{\max}(u_0,v_0).
$$
where  $q^*=q_{\chi_1,\chi_2}$, $\tau_0\in [0,T_{\max}(u_0,v_0))$ is as  in {\bf (H2)}, and
$$
m^*= \frac{a_{\max}|\Omega|}{\min\{b_{\min},c_{\min}\}}.
$$ Let
$$
 \delta^*=
\frac{\delta_0 |\Omega| (a_{\min}-{ \chi^{**}(\mu,\chi_1,\chi_2)})^{\frac{1}{q^*}}\big(\min\{b_{\min},c_{\min}\}\big)^{\frac{1}{q^*}-1} }{(b_{\max}+c_{\max})^{\frac{1}{q^*}}\big(a_{\max}\big)^{\frac{1}{q^*}-1} }.
$$
Then by  Lemma \ref{pre-lm-1} and \eqref{lower-bound-u+v-eqq0},  we have
\begin{align*}
\inf_{x\in\Omega} w(t,x)&\ge \delta_0\int_\Omega (u+v)\ge \frac{\delta_0 |\Omega|^{\frac{q^*+1}{q^*}}}{\Big(\int_\Omega (u+v)^{-q^*}\Big)^{\frac{1}{q^*}}}\ge \delta^*{ C_{\chi_1,\chi_2}}\quad\forall\, \tau_0\le t<T_{\max}(u_0,v_0).
\end{align*}
By Lemma \ref{new-pre-lm1},  for any $\varepsilon_1>0$, there is $C(\varepsilon_1,p)>0$ such that
$$
\int_\Omega w^{p+1}\le\varepsilon_1 \int_\Omega (u+v)^{p+1}+C(\varepsilon_1,p)\Big(\int_\Omega (u+v)\Big)^{p+1}  { \quad \text{for all}\;\; 0< t<T_{\max}(u_0,v_0).}
$$
By Young's inequality,  any $\varepsilon_2>0$, there is $\tilde C(\varepsilon_2,p)>0$ such that
$$
\Big(\int_\Omega (u+v)^{p+1}\Big)^{\frac{p}{p+1}} \Big(\int_\Omega (u+v)\Big)\le
\varepsilon_2\int_\Omega (u+v)^{p+1}+\tilde C(\varepsilon_2,p)\Big(\int_\Omega (u+v)\Big)^{p+1}.
$$
Then, by \eqref{aux-u+v-eq2}, we  have
\begin{align*}
\frac{1}{p}\frac{d}{dt}\int_\Omega (u+v)^{p}
&\le \frac{\tilde \chi { (p-1)}\big(C^*\big)^{\frac{1}{p+1}}{ \max\{\nu,\lambda\}}}{2\delta^*{ C_{\chi_1,\chi_2}}}\int_\Omega (u+v)^{p+1}
+\frac{\tilde  \chi {(p-1)} \big(C^{**}\big)^{\frac{1}{p+1}}\varepsilon_1^{\frac{1}{p+1}}}{2\delta^*{C_{\chi_1,\chi_2}}}\int_\Omega (u+v)^{p+1}\nonumber\\
&\quad  +\frac{ \tilde \chi { (p-1)} \big(C^{**}\big)^{\frac{1}{p+1}}\big(C(\varepsilon_1,p)\big)^{\frac{1}{p+1}}}{2\delta^*{C_{\chi_1,\chi_2}}}\Big(\int_\Omega (u+v)^{p+1}\Big)^{\frac{p}{p+1}}\Big(\int_\Omega  (u+v)\Big)\nonumber\\
&\quad +a_{\max} \int_\Omega (u+v)^{p} - \min\{b_{\min},c_{\min}\} \int_\Omega (u+v)^{p+1}\nonumber\\
&\le \frac{\tilde \chi { (p-1)} \big(C^*\big)^{\frac{1}{p+1}}{\max\{\nu,\lambda\}}}{2\delta^*{ C_{\chi_1,\chi_2}}}\int_\Omega (u+v)^{p+1}\nonumber\\
&\quad
+\frac{\tilde \chi { (p-1)}\big(C^{**}\big)^{\frac{1}{p+1}}\big(\varepsilon_1^{\frac{1}{p+1}}+\big(C(\varepsilon_1,p)\big)^{\frac{1}{p+1}}\varepsilon_2\big)
}{2\delta^*{ C_{\chi_1,\chi_2}}}\int_\Omega (u+v)^{p+1}\nonumber\\
&\quad +\frac{\tilde  \chi { (p-1)} \big(C^{**}\big)^{\frac{1}{p+1}}\big(C(\varepsilon_1,p)\big)^{\frac{1}{p+1}}\tilde C(\varepsilon_2,p)}{2\delta^*{C_{\chi_1,\chi_2}}}\Big(\int_\Omega (u+v)\Big)^{p+1}\nonumber\\
&\quad +a_{\max} \int_\Omega (u+v)^{p} - \min\{b_{\min},c_{\min}\} \int_\Omega (u+v)^{p+1}.
\end{align*}

Now,
by {\bf (H1)}, there are  $p^*>\max\{2,N\}$, $\varepsilon_1^*>0$ and
$\varepsilon_2^*>0$ such that
$$
\min\big\{b_{\min},c_{\min}\}>\frac{\tilde \chi { (p-1)}\big(C^*\big)^{\frac{1}{p^*+1}}{ \max\{\nu,\lambda\}}}{2\delta^*{ C_{\chi_1,\chi_2}}}+\frac{\tilde \chi { (p-1)}\big(C^{**}\big)^{\frac{1}{p^*+1}}\big((\varepsilon_1^*)^{\frac{1}{p^*+1}}+\big(C(\varepsilon_1^*,p^*)\big)^{\frac{1}{p^*+1}}\varepsilon_2^*\big)
}{2\delta^*{ C_{\chi_1,\chi_2}}}.
$$
Therefore, there is $b^*>0$ such that
\begin{align*}
\frac{1}{p^*}\frac{d}{dt}\int_\Omega(u+v)^{p*}\le& \frac{\tilde \chi{ (p-1)} \big(C^{**}\big)^{\frac{1}{p+1}}\big(C(\varepsilon_1^*,p^*)\big)^{\frac{1}{p^*+1}}\tilde C(\varepsilon_2^*,p^*)}{2\delta^*{ C_{\chi_1,\chi_2}}}\Big(\int_\Omega (u+v)\Big)^{p^*+1}\\
&+a_{\max} \int_\Omega(u+v)^{p^*}-b^*\int_\Omega(u+v)^{p^*+1}\quad \forall\, \tau_0\le t<T_{\max}(u_0,v_0).
\end{align*}
This implies that there is $M_1'>0$ independent of $(u_0,v_0)$  such that
$$
\frac{d}{dt}\int_\Omega (u+v)^{p^*}\le -\int_\Omega (u+v)^{p^*}+M_1'\quad \forall\,\, \tau_0<t<T_{\max}(u_0,v_0).
$$
Hence
\begin{equation}
\label{new-lp-bound-eq1}
\int_\Omega (u+v)^p(t,x;u_0,v_0)\le e^{-(t-\tau_0)}\int_\Omega (u+v)(\tau_0,x;u_0,v_0)+M_1'\quad \forall\, \tau_0\le t<T_{\max}(u_0,v_0).
\end{equation}
Therefore
Theorem \ref{new-main-thm2}(1) holds with $p=p^*$, $M_1=M_1'$,  and any $\tilde p^*>1$.}

It remains to prove that Theorem \ref{new-main-thm2}(1) holds with $p=2p^*$.
To this end, let $\tilde p=2p^*$.
By \eqref{new-new-u+v-eq1}, we have
\begin{align}
\label{upper-bound-u+v-eqqq2}
\frac{1}{\tilde p}\frac{d}{dt}\int_\Omega (u+v)^{\tilde p}
\le &  -(\tilde p-1)\int_\Omega (u+v)^{\tilde p-2}|\nabla (u+v)|^2-\int_\Omega (u+v)^{\tilde p}\nonumber\\
& +(\tilde p-1)\int_\Omega (u+v)^{\tilde p-2}\frac{\chi_1 u+\chi_2 v}{w} \nabla(u+v)\cdot\nabla w\nonumber\\
& +\big(a_{\max}+1\big) \int_\Omega (u+v)^{\tilde p} - \min\{b_{\min},c_{\min}\} \int_\Omega (u+v)^{\tilde p+1}.
\end{align}
Note that
\begin{equation}
\label{new-added--eq1}
\int_\Omega (u+v)^{\tilde p}\le \frac{\tilde p}{\tilde p+1}\int_\Omega (u+v)^{\tilde p+1}
+\frac{|\Omega|}{\tilde p+1}.
\end{equation}
Note also that
\begin{align}
\label{new-added-eq2}
&\int_\Omega (u+v)^{\tilde p-2}\frac{\chi_1 u+\chi_2 v}{w} \nabla(u+v)\cdot\nabla w\nonumber\\
&\le
(\chi_1+\chi_2)\int_\Omega\frac{ (u+v)^{\tilde p-1}}{w}|\nabla(u+v)| \, |\nabla w|\nonumber\\
&\le \frac{1}{2}\int_\Omega (u+v)^{\tilde p-2}|\nabla (u+v)|^2 +\frac{(\chi_1+\chi_2)^2}{2}\int_\Omega (u+v)^{\tilde p}\frac{|\nabla w|^2}{w^2}\nonumber\\
&\le  \frac{1}{2}\int_\Omega (u+v)^{\tilde p-2}|\nabla (u+v)|^2 +\frac{(\chi_1+\chi_2)^2}{2}\frac{\tilde p+1}{\tilde p}\int_\Omega (u+v)^{\tilde p+1}
\nonumber\\
&\quad + \frac{(\chi_1+\chi_2)^2}{2}\frac{1}{\tilde p+1}\int_\Omega\frac{|\nabla w|^{2\tilde p+2}}{w^{2\tilde p+2}}.
\end{align}
By Proposition  \ref{new-main-prop2}, we have
\begin{equation}
\label{new-added-eq3}
\int_\Omega \frac{|\nabla w|^{2\tilde p+2}}{w^{2\tilde p +2}}\le
\frac{1}{{\Big(\inf_{x\in\Omega}w(t,x)\Big)^{\tilde p+1}}}\Big( M(\tilde p+1,\tilde p+1)\int_\Omega (u+v)^{\tilde p+1}+\tilde M(\tilde p+1,\tilde p+1)\int_\Omega w^{\tilde p+1}\Big)
\end{equation}
By  {\bf (H2)}  and Theorem \ref{new-main-thm4}(2),
\begin{equation*}
\int_\Omega(u+v)^{-q}\le  e^{-(a_{\min}-\chi^{**}(\mu,\chi_1,\chi_2)) {q }(t-\tau)} \int_\Omega(u(x,\tau_0;u_0,v_0)+v(x,\tau_0;u_0,v_0))^{-{ q}}dx+ \tilde   C_1
\end{equation*}
for all  $\tau_0<t<T_{\max}(u_0,v_0)$,
where { $q=q_{\chi_1,\chi_2}$ and}
\begin{equation*}
\tilde C_1=
\frac{(b_{\max}+c_{\max})|\Omega|}{a_{\min}-\chi^{**}(\mu,\chi_1,\chi_2)}\Big(\frac{a_{\max}}{\min\{b_{\min},c_{\min}\}}\Big)^{1-q}.
\end{equation*}
This together with \eqref{lower-bound-u+v-eqq0}   implies that
\begin{align}
\label{new-added-eq4}
\frac{1}{\inf_{x\in\Omega} w(t,x)}&\le \frac{1}{\delta_0}\frac{1}{\int_\Omega (u(t,x;u_0,v_0)+v(t,x;u_0,v_0))}\nonumber\\
&\le \frac{1}{\delta_0 |\Omega|^{\frac{q+1}{q}}} \Big(\int_\Omega(u(t,x;u_0,v_0)+v(t,x;u_0,v_0))^{-q}\Big)^{\frac{1}{q}}\nonumber\\
&\le  \frac{1}{\delta_0|\Omega|^{q+1}{q}}\Big( e^{-(a_{\min}-\chi^{**}(\mu,\chi_1,\chi_2)) {q }(t-\tau)} \int_\Omega(u(x,\tau_0;u_0,v_0)+v(x,\tau_0;u_0,v_0))^{-{ q}}dx+ \tilde   C_1\Big)
\end{align}
for $\tau_0<t<T_{\max}(u_0,v_0)$.
By Lemma \ref{pre-lm-4}, there is $\tilde C_2>0$ such that
\begin{equation}
\label{upper-bound-u+v-eqqq7}
\int_\Omega w^{\tilde p+1}\le \tilde C_2\int_\Omega (u+v)^{\tilde p+1}.
\end{equation}
By {\bf (H2)} and \eqref{upper-bound-u+v-eqqq2}-\eqref{upper-bound-u+v-eqqq7},
 there is $\tilde C_2>0$  independent of $(u_0,v_0)$ such that
\begin{equation*}
\frac{1}{\tilde p}\cdot\frac{d}{dt}\int_\Omega (u+v)^{\tilde p}\le -\frac{\tilde p-1}{2}\int_\Omega (u+v)^{\tilde p-2}|\nabla (u+v)|^2 -\frac{1}{p}\int_\Omega (u+v)^{\tilde p}+\tilde C_2\int_\Omega (u+v)^{\tilde p+1}+\tilde C_2
\end{equation*}
for all $\tau_0<t<T_{|max}(u_0,v_0)$.
By    Gagliardo-Nirenberg inequality,  there are  $\tilde C_3>0$, $\tilde C_4>0$, and $\tilde C_5>0$ such that
\begin{align*}
\int_\Omega (u+v)^{\tilde p+1}&=\int_\Omega\big((u+v)^{\frac{\tilde p}{2}}\Big)^{\frac{2(\tilde p+1)}{\tilde p}}=\|(u+v)^{\frac{\tilde p}{2}}\|_{L^{\frac{2(\tilde p+1)}{\tilde p}}}^{\frac{2(\tilde p+1)}{\tilde p}}\\
&\le \Big( \tilde C_3\|\nabla (u+v)^{\frac{\tilde p}{2}}\|_{L^2}^\theta \|(u+v)^{\frac{\tilde p}{2}}\|_{L^1}^{1-\theta}
+\tilde C_3\|(u+v)^{\frac{\tilde p}{2}}\|_{L^1}\Big)^{\frac{2(\tilde p+1)}{\tilde p}}\\
&\le\tilde  C_4 \|\nabla (u+v)^{\frac{\tilde p}{2}}\|_{L^2}^{\theta  \frac{2(\tilde p+1)}{\tilde p}} \|(u+v)^{\frac{\tilde p}{2}}\|_{L^1}^{(1-\theta)  \frac{2(\tilde p+1)}{\tilde p}}
+\tilde C_4\|(u+v)^{\frac{\tilde p}{2}}\|_{L^1}^{\frac{2(\tilde p+1)}{\tilde p}}\\
&\le \tilde C_5 \Big(\int_\Omega u^{\tilde p-2}|\nabla (u+v)|^2 \Big)^{\theta \frac{\tilde p+1}{\tilde p}}
\Big(\int_\Omega (u+v)^{\frac{\tilde p}{2}}\Big)^{(1-\theta)\frac{2(\tilde p+1)}{\tilde p}} +\tilde C_4 \Big(\int_\Omega (u+v)^{\frac{\tilde p}{2}}\Big)^{\frac {2(\tilde p+1)}{\tilde p}},
\end{align*}
where
\begin{equation*}
    \theta=\frac{1-\frac{\tilde p}{2(\tilde p+1)}}{\frac{1}{2}+\frac{1}{N}}=\frac{\frac{1}{2}+\frac{1}{2(\tilde p+1)}}{\frac{1}{2}+\frac{1}{N}}=\frac{\tilde p+2}{\tilde p+1}\cdot\frac{N}{N+2} \in (0,1),
\end{equation*}
and
$$
\frac{\theta (\tilde p+1)}{\tilde p}=\frac{\tilde p+2}{\tilde p+1}\cdot\frac{N}{N+2}\cdot\frac{\tilde p+1}{\tilde p}=
\frac{\tilde p N+2N}{\tilde p N+2\tilde p}<1\quad {\rm and}\quad \tilde p=2p^*.
$$
Then by Young's inequality,
 for any $0<\epsilon<\frac{\tilde p-1}{2}$, there is $\tilde C_\epsilon>0$  such that
$$
\int_\Omega (u+v)^{\tilde p+1}\le \epsilon \int_\Omega (u+v)^{\tilde p-2}|\nabla (u+v)|^2+
\tilde C_\epsilon\Big(1+\int_\Omega (u+v)^{p^*}\Big)^{\tilde p^*},
$$
where $\tilde p^*=\frac{2(1-\theta)(\tilde p+1)}{{ p-\theta(p+1)}}$.
Hence  there is $\tilde C_6>0$ independent of $(u_0,v_0)$ such that
$$
\frac{d}{dt}\int_\Omega (u+v)^{\tilde p}\le  -\int_\Omega (u+v)^{\tilde p}+\tilde C_\epsilon \tilde C_2 \Big(\int_\Omega (u+v)^{p^*}\Big)^{\tilde p^*}+\tilde C_6\quad \forall\, \tau_0<t<T_{\max}(u_0,v_0).
$$
This together with \eqref{new-lp-bound-eq1} implies that
there is $M_1''>0$ independent of $(u_0,v_0)$  such that
\begin{align}
\label{upper-bound-u+v-eqqq1}
&\int_\Omega (u+v)^{p}(t,x;u_0,v_0)dx\nonumber\\
&\le  e^{-(t-{ \tau_0}) }\int_\Omega (u+v)^{p}(\tau_0,x;u_0,v_0)dx \nonumber\\
& \quad +M_1'' e^{-(t-\tau_0)}(t-\tau_0)\Big(\int_\Omega(u+v)^{p^*}(\tau_0,x;u_0,v_0)\Big)^{\tilde p^*}+M_1'',\,\, \forall\, \tau_0<t<T_{\max}(u_0,v_0).
\end{align}
Theorem \ref{new-main-thm2}(1) then follows with $M_1=\max\{M_1',M_1''\}$.
\end{proof}

\subsection{Proof of Theorem \ref{new-main-thm2}(2)}

In this subsection, we prove Theorem \ref{new-main-thm2}(2).

\begin{proof}[Proof of Theorem \ref{new-main-thm2}(2)]
  For given $p>2N$ and $0<\theta<1-\frac{2N}{p}$, choose   $\beta\in (0,\frac{1}{2})$ such that $2\beta-\frac{2N}{p}>\theta.$ By Lemma \ref{pre-lm-2}, there is $C_{p,\beta}>0$ such that for any $u_0, v_0$ satisfying \eqref{initial-cond-eq},
\begin{equation}
\label{new-thm2-eq1}
\|u(\cdot,\cdot)+v(\cdot,\cdot)\|_{C^\theta(\bar\Omega)}\le C_{p,\beta} \|u(t,\cdot)+v(t,\cdot)\|_{X_{\frac{p}{2}}^\beta}\quad \forall\, 0<t<T_{\max}(u_0,v_0).
\end{equation}
Note that in view of \eqref{new-u+v-eq1-0}, we have that
\begin{align}
\label{new-thm2-eq2}
    \|u(t,\cdot)+v(t,\cdot)\|_{X_{\frac{p}{2}}^{\beta}} & \leq   \|A_{\frac{p}{2}}^{\beta}e^{-A_{\frac{p}{2}}(t-\tau)} (u(\tau,\cdot;u_0,v_0)+v(\tau,\cdot;u_0,v_0))\|_{L^{p/2}(\Omega)}\nonumber\\
   &\quad + \int_{\tau}^{t}\left\| A_{\frac{p}{2}}^{\beta}e^{-A_{\frac{p}{2}}(t-s)}\nabla\cdot \left(\frac{\chi_1 u(s,\cdot)+ \chi_2 v(s,\cdot) }{w(s,\cdot)}  \nabla w(s,\cdot) \right)\right\|_{L^{p/2}(\Omega)}ds\nonumber \\
    & \quad + \int_{\tau}^{t}\| A_{\frac{p}{2}}^{\beta}e^{-A_{\frac{p}{2}}(t-s)} g (u(s,\cdot),v(s,\cdot))\|_{L^{p/2}(\Omega)}ds
\end{align}
for any $0\le \tau <T_{\max}(u_0,v_0)$ and $\tau<t<T_{\max}(u_0,v_0)$,
where   $A_{\frac{p}{2}}=-\Delta+\mu I$ is defined as in \eqref{A-p-def} with $p$ being replaced by $\frac{p}{2}$, and
$$g (u,v)=(\mu+a_1)u+(\mu+a_2 )v -(b_1u^2+b_2v^2)-(c_1+c_2)uv.  $$
By Lemma \ref{pre-lm-2} again,  there is $\gamma>0$ such that
\begin{align}
\label{new-thm2-eq3}
 &\|A_{\frac{p}{2}}^{\beta}e^{-A_{\frac{p}{2}}(t-\tau)} (u(\tau,\cdot;u_0,v_0)+v(\tau,\cdot;u_0,v_0))\|_{L^{p/2}(\Omega)}\nonumber\\
&\le C_{p,\beta} (t-\tau)^{-\beta } e^{-\gamma (t-\tau)}\|u(\tau,\cdot;u_0,v_0)+v(\tau,\cdot;u_0, v_0)\|_{L^{p/2}}\quad \forall\, {\tau<t<T_{\max}(u_0,v_0)}.
\end{align}
By Lemma \ref{pre-lm-3} and  Lemma \ref{pre-lm-4},   for any $\epsilon>0$,   there are  $\gamma>0$, $C_p>0$ and $C_{p,\beta,\epsilon}>0$ such that
\begin{align}
\label{new-thm2-eq4}
&\int_{\tau }^{t}\left\| A_{\frac{p}{2}}^{\beta}e^{-A+{\frac{p}{2}}(t-s)}\nabla\cdot \left(\frac{\chi_1 u(s,\cdot)+ \chi_2 v(s,\cdot) }{w(s,\cdot)}  \nabla w(s,\cdot) \right)\right\|_{L^{p/2}(\Omega)}ds\nonumber \\
& \leq  C_{p,\beta,\epsilon}  \int_\tau^{t} (1+ (t-s)^{-\beta-\frac{1}{2}-\epsilon})e^{-\gamma(t-s)}  \left\| \frac{\chi_1u(s,\cdot)+\chi_2 v(s,\cdot)}{w(s,\cdot)} \cdot \nabla w(s,\cdot) \right\|_{L^{p/2}(\Omega)} ds\nonumber \\
&\le {C_{p,\beta,\epsilon} (\chi_1+\chi_2)} \int_\tau^{t}  \frac{ (1+(t-s)^{-\beta-\frac{1}{2}-\epsilon})e^{-\gamma(t-s)}  \left\| u(s,\cdot)+v(s,\cdot)\right\|_{L^p(\Omega)}\left\| \nabla w(s,\cdot) \right\|_{L^{p}(\Omega)} } {\displaystyle \inf_{\tau\le t<\min\{T,T_{\max}\},x\in\Omega} w(t,x)} ds\nonumber\\
&\le {C_p C_{p,\beta,\epsilon} (\chi_1+\chi_2)} \int_\tau^{t}  \frac{ (1+(t-s)^{-\beta-\frac{1}{2}-\epsilon})e^{-\gamma(t-s)}  \left\| u(s,\cdot)+v(s,\cdot)\right\|^2_{L^p(\Omega)}} {\displaystyle \inf_{\tau\le t<\min\{T,T_{\max}\},x\in\Omega} w(t,x)} ds.
\end{align}
By Lemma \ref{pre-lm-2} again,
\begin{align}
\label{new-thm2-eq5}
& \int_{\tau}^{t}\| A_{\frac{p}{2}}^{\beta}e^{-A_{\frac{p}{2}}(t-s)} g(u(s,\cdot),v(s,\cdot))\|_{L^{p/2}(\Omega)}ds\nonumber \\
     & \le C_{p,\beta} \int_\tau^{t} (t-s)^{-\beta}e^{-\gamma(t-s)}   \Big\{\mu |\Omega|^{2/p}+a_{\max} \|u(s,\cdot)+v(s,\cdot)\|_{L^{p/2}}+b_{\max}\|(u(s,\cdot)+v(s,\cdot))^2\|_{L^{p/2}}\nonumber\\
     & \quad \quad\quad \quad\quad \quad\quad \quad \quad \quad + \frac{c_{\max}}{2}\|(u(s,\cdot)+v(s,\cdot))^2  \|_{L^{p/2}}\Big\}  ds.
\end{align}
\eqref{new-infinity-bdd-eq0} then follows from \eqref{new-thm2-eq1} to \eqref{new-thm2-eq5}.
\end{proof}

\section{Proofs of Theorems \ref{new-main-thm3} and \ref{new-main-thm5}}

In this section, we study  the global existence and   uniform boundedness and  uniform  pointwise persistence of classical solutions of \eqref{main-eq}, and
 prove Theorems \ref{new-main-thm3} and \ref{new-main-thm5}.

Throughout this section, for given $u_0,v_0$ satisfying   \eqref{initial-cond-eq}, we put
$$
(u(t,x),v(t,x),w(t,x)):=(u(t,x;u_0,v_0),v(t,x;u_0,v_0),w(t,x;u_0,v_0)),
$$
and if no confusion occurs, we may drop $(t,x)$ in $u(t,x)$ (resp. $v(t,x)$, $w(t,x)$).

\subsection{Global existence and proof of Theorem \ref{new-main-thm3}}

In this subsection, we study  the global existence of classical solutions of \eqref{main-eq} and prove Theorem \ref{new-main-thm3}.

\begin{proof}[Proof of Theorem \ref{new-main-thm3}]
We prove that $T_{\max}=\infty$ by contradiction.
 Assume that $T_{\max}<\infty$.  Then by Lemma \ref{pre-lm-1} and
Theorem \ref{new-main-thm1},
there is $\delta>0$ such that
\begin{equation}
\label{delta-lower-bd-eq1}
    w(t,x) \ge \delta \quad \quad \text{\rm for all}\,\, x\in \Omega \quad \text{\rm and}\,\, t\in (0,T_{\max}),
\end{equation}
and then  by  Proposition \ref{local-existence-prop}, we have
\begin{equation*}
\limsup_{t \nearrow T_{\max}} \left\| u(t,\cdot) + v(t,\cdot) \right\|_{C^0(\bar \Omega)}  =\infty.
\end{equation*}
But,  { by \eqref{upper-bound-u+v-eqqq1},  \eqref{delta-lower-bd-eq1},  and Theorem \ref{new-main-thm2}(2)},  we have
$$
\lim_{t\nearrow T_{\max}}  \left\| u(t,\cdot) + v(t,\cdot) \right\|_{C^0(\bar \Omega)}  <\infty,
$$
which is a contradiction. Therefore $T_{\max}=\infty$.

\end{proof}

\subsection{Proof of Theorem \ref{new-main-thm5}}

In this subsection, we investigate the ultimate upper bounds of $\int_\Omega (u+v)^{-q}$,
$\int_\Omega (u+v)^p$, and $\|u+v\|_\infty$  for  globally defined classical solutions of \eqref{main-eq}, ultimate lower bound of $\inf_{x\in\Omega} (u+v)$,  and prove Theorem \ref{new-main-thm5}.

\begin{proof} [Proof of Theorem \ref{new-main-thm5}]

{ First, we prove \eqref{new-uniform-bd-eq1}.  By Theorem \ref{new-main-thm3},
$T_{\max}(u_0,v_0)=\infty$.
By the assumption {\bf (H1)}, we have $a_{\min}>\chi^{**}(\mu,\chi_1,\chi_2)$.
By {\bf (H2)} and  Theorem \ref{new-main-thm4}(2),
\begin{align}
\label{uniform-bd-proof-eq1-2}
&\int_\Omega (u(x,t;u_0,v_0)+v(x,t;u_0,v_0))^{-q}dx\nonumber\\
& \le e^{-(a_{\min}-\chi^{**}(\mu,\chi_1,\chi_2)) (t-\tau_0)} \int_\Omega(u(\tau_0,x;u_0,v_0)+v(\tau_0,x;u_0,v_0))^{-q}dx+ \tilde   C_{\tau_0,u_0,v_0},
\end{align}
for all $\tau_0<t<\infty$, where $q=q_{\chi_1,\chi_2} (\le 1)$,  and
\begin{equation*}
\tilde C_{\tau,u_0,v_0} = \frac{(b_{\max}+c_{\max})|\Omega|}{a_{\min}-\chi^{**}(\mu,\chi_1,\chi_2)} \times
\Big(\frac{a_{\max}}{\min\{b_{\min},c_{\min}\}}\Big)^{1-q}.
\end{equation*}
This together with \eqref{uniform-bd-proof-eq1-2} implies that \eqref{new-uniform-bd-eq1}
holds with $M_1^*=\frac{(b_{\max}+c_{\max})|\Omega|}{a_{\min}-\chi^{**}(\mu,\chi_1,\chi_2)} \times
\Big(\frac{a_{\max}}{\min\{b_{\min},c_{\min}\}}\Big)^{1-q}$.

\medskip

Next,  it is clear that  \eqref{new-uniform-bd-eq2} with $M_2^*=M_1$ follows from Theorem \ref{new-main-thm2}(1).

\smallskip

{ Now, we prove \eqref{new-uniform-bd-eq3}. By \eqref{w-lower-bound-eq1} and \eqref{lower-bound-u+v-eqq0},
we have
\begin{align}
\label{w-low-bdd}
w(t,x;u_0,v_0)&\ge  \delta_0\int_\Omega(u+v)(t,x;u_0,v_0)\ge \frac{\delta_0 |\Omega|^{\frac{q+1}{q}}}{\Big(\int_\Omega (u+v)^{-q}(t,x;u_0,v_0)\Big)^{\frac{1}{q}}}
\end{align}
for all $0<t<\infty$. Note that $0<\theta<\frac{N}{p^*}=\frac{2N}{2p^*}$.
By Theorem \ref{new-main-thm2}(2), we have
\begin{align*}
&\|u(t,\cdot;u_0,v_0)+v(t,\cdot;u_0,v_0)\|_{C^\theta(\bar\Omega)}\nonumber\\
&\le  M_2 \Big[ (t-\tau)^{-\beta} e^{-\gamma (t-\tau)} \|u(\tau,\cdot;u_0,v_0)+v(\tau,\cdot;u_0,v_0)\|_{L^{2 p^*}}\nonumber\\
&\,\,\,\,+\frac{\displaystyle \sup_{\tau \le t <\hat T}\|u(t,\cdot;u_0,v_0)+
v(t,\cdot;u_0,v_0)\|_{L^{2 p^*}}^2}{\displaystyle \inf_{\tau \le t< \hat T, x\in\Omega} w(t,x;u_0,v_0)}\nonumber\\
&\,\,  \,\,+\sup_{\tau \le t<\hat T}\|u(t,\cdot;u_0,v_0)+v(t,\cdot;u_0,v_0)\|_{L^{2 p^*}}+1\Big]
\end{align*}
for any $0<\tau<t<\hat T<\infty $.
By Theorem \ref{new-main-thm2}(1)  and \eqref{w-low-bdd},
and letting $\hat T-\tau\to\infty$ and
$\tau\to\infty$,  we obtain \eqref{new-uniform-bd-eq3} with
$M_3^*=M_1\Big[\frac{M_1^{\frac{1}{p^*}} (M_1^*)^{\frac{1}{q}}}{\delta_0|\Omega|^{\frac{q+1}{q}}}+M_2^{\frac{1}{2p^*}}+1\Big]$.}}

\medskip

(2) We prove  it by contradiction.  Assume that there is no $M_0^*>0$ such that \eqref{new-uniform-bd-eq4} holds.
Then for any $n\in\N$,  there are $u_n,v_n$ satisfying {\bf (H2)} such that
\begin{equation}
\label{new-set-eq7}
\liminf_{t\to\infty} \inf_{x\in\Omega}(u(t,x;0,u_n,v_n)+v(t,x;0,u_n,v_n))\le m_n:=\frac{1}{n}.
\end{equation}

By  { \eqref{new-uniform-bd-eq1}, \eqref{new-uniform-bd-eq3},  and \eqref{w-low-bdd}}, there are $T_n>0$
and $\tilde M_1^*$  such that
\begin{equation}
\label{new-set-eq9}
\int_\Omega u(t,x;u_n,v_n)+v(t,x;u_n,v_n)dx\ge {\tilde M_1^*}\quad \forall\, t\ge T_n
\end{equation}
and
\begin{equation}
\label{new-set-eq8}
\|u(t,x;u_n,v_n)+v(t,x;u_n,v_n)\|_{C^\theta(\bar\Omega)}\le 2M^*_3\quad \forall \, t\ge T_n.
\end{equation}
By \eqref{new-set-eq7}, there are $t_n\in\R$ with $t_n\ge T_n+1$ and $x_n\in\bar\Omega$ such that
\begin{equation}
\label{new-set-eq10}
u(t_n,x_n;u_n,v_n)+v(t_n,x_n;u_n,v_n)\le \frac{2}{n}.
\end{equation}
Note that
\begin{align}
\label{new-set-eq11}
u(t_n,x;u_n,v_n)&=u(1,x;u(t_n-1,\cdot;u_n,v_n), v(t_n-1, u_n,v_n))
\end{align}
and
\begin{align}
\label{new-set-eq11*}
v(t_n,x;0,u_n,v_n)&=v(1,x;u(t_n-1,\cdot;u_n,v_n),v(t_n-1,\cdot;u_n,v_n)).
\end{align}
By \eqref{new-set-eq8}, without loss of generality, we may assume that there are $u_0^*,v_0^*,u_1^*,v_1^*\in C(\bar\Omega)$ such that
\begin{equation*}
    \lim_{n\to\infty} u(t_n-1,x;u_n,v_n)=u_0^*(x),\quad   \lim_{n\to\infty} v(t_n-1,x;u_n,v_n))=v_0^*(x),
\end{equation*}
and
\begin{equation*}
    \lim_{n\to\infty} u(t_n,x;u_n,v_n)=u_1^*(x),\quad \lim_{n\to\infty} v(t_n,x;u_n,v_n))=v_1^*(x).
\end{equation*}
uniformly in $x\in\bar\Omega.$ Then, by \eqref{new-set-eq11} and \eqref{new-set-eq11*}, we have that
\begin{equation}
\label{new-set-eq12}
u_1^*(x)+v_1^*(x)=u(1,x;u_0^*,v_0^*)+v(1,x;u_0^*,v_0^*).
\end{equation}

By \eqref{new-set-eq9}, \eqref{new-set-eq8}, and the Dominated Convergence Theorem,
$$
\int_\Omega u_0^*(x)+v_0^*(x)=\lim_{n\to\infty}\int_\Omega (u(t_n-1,x;0,u_n,v_n)+v(t_n-1,x;0,u_n,v_n))dx\ge {\tilde M_1^*}.
$$
Then by \eqref{new-set-eq12} and comparison principles for parabolic equations,
$$
\inf_{x\in\bar\Omega} (u_1^*(x)+v_1^*(x))>0.
$$
This together with \eqref{new-set-eq10}  implies that
$$
\frac{2}{n}\ge \inf_{x\in \Omega} (u(t_n,x;0,u_n,v_n)+v(t_n,x;0,u_n,v_n))\ge \frac{1}{2}\inf_{x\in\Omega} (u_1^*(x)+v_1^*(x))\quad \forall\, n\gg 1,
$$
which is a contradiction.
Therefore (2) holds.
\end{proof}

\section{Appendix: {Proof of Proposition  \ref{new-main-prop2}}}

{   Proposition \ref{new-main-prop2} can be proved by the similar arguments as those in
\cite[Proposition 4.2]{HKWS}.  For the reader's convenience, we provide a  proof of Proposition \ref{new-main-prop2} in  this section.}

Throughout this section, for given $u_0,v_0$ satisfying   \eqref{initial-cond-eq}, we put
$$
(u(t,x),v(t,x),w(t,x)):=(u(t,x;u_0,v_0),v(t,x;u_0,v_0),w(t,x;u_0,v_0)),
$$
and if no confusion occurs, we may drop $(t,x)$ in $u(t,x)$ (resp. $v(t,x)$, $w(t,x)$).

We first present some lemmas.

\begin{lemma}
\label{lem-4-5}
Let $p \ge 3$  and $k \ge 2.$  Then
\begin{equation}
    \label{aaux-new-eq00}
    (k-1)\int_{\Omega}\frac{|\nabla w|^{2p}}{w^{k}} \leq (p-1) \int_{\Omega} \frac{|\nabla w|^{2p-4}}{w^{k-1}} \nabla w \cdot \nabla |\nabla w|^2 +  \mu \int_{\Omega}\frac{|\nabla w|^{2p-2}}{w^{k-2}}
\end{equation}
{for all $t \in (0, T_{\rm max}(u_0,v_0))$}.
\end{lemma}

\begin{proof}
Multiplying  the third equation in \eqref{main-eq} by $\frac{|\nabla w|^{2p-2}}{w^{k-1}}$ and integrating it over  $\Omega$ yields that
\begin{align*}
  & (k-1) \int_{\Omega} \frac{|\nabla w|^{2p}}{w^{k}} + \nu \int_{\Omega}\frac{u|\nabla w|^{2p-2}}{w^{k-1}} + \lambda  \int_{\Omega}\frac{v|\nabla w|^{2p-2}}{w^{k-1}}\\
&= (p-1) \int_{\Omega} \frac{|\nabla w|^{2p-4}}{w^{k-1}} \nabla w \cdot \nabla |\nabla w|^2 + \mu \int_{\Omega}\frac{|\nabla w|^{2p-2}}{w^{k-2}}
\end{align*}
for all $t\in (0,T_{\max}(u_0,v_0))$. This implies \eqref{aaux-new-eq00}.
\end{proof}

To prove Proposition \ref{new-main-prop2}, it is then essential to provide proper estimate for the integral $\int_{\Omega} \frac{|\nabla w|^{2p-4}}{w^{k-1}} \nabla w \cdot \nabla |\nabla w|^2$.
By Young's inequality,  for any $0<\varepsilon<p-\frac{3}{2}$, we have
 \begin{align}
\label{aaux-new-eq1}
\int_{\Omega} \frac{ |\nabla w|^{2p-4}}{w^{k-1}} \nabla w \cdot \nabla |\nabla w|^2 &\leq \frac
{p-\frac{3}{2}-\varepsilon}{ p+k-3}\int_{\Omega} \frac{|\nabla w|^{2p-6}}{w^{k-2}}\big |\nabla |\nabla w|^2\big|^2 \nonumber\\
    & \quad + \frac{p+k-3}{4\left(p-\frac{3}{2}-\varepsilon\right)} \int_{\Omega}\frac{|\nabla w|^{2p}}{w^{k}}.
\end{align}
By a direct calculation, we have
\begin{align}
\label{aaux-new-eq2}
    (p-2) \int_{\Omega} \frac{|\nabla w|^{2p-6}}{w^{k-2}}\big |\nabla |\nabla w|^2\big|^2 &= \int_{\partial \Omega} \frac{|\nabla w|^{2p-4}}{w^{k-2}} \frac{\partial |\nabla w|^2}{\partial \nu} + (k-2)\int_{\Omega} \frac{ |\nabla w|^{2p-4}}{w^{k-1}} \nabla w \cdot \nabla |\nabla w|^2 \nonumber\\
    & \quad -  \int_{\Omega} \frac{|\nabla w|^{2p-4}}{w^{k-2}}  \Delta |\nabla w|^2.
\end{align}
Note that
$$\Delta |\nabla w|^2=2 \nabla w \cdot \nabla(\Delta w)+{ 2|D^2 w|^2}= 2 \nabla w \cdot \nabla(\mu w-\nu u -\lambda v)+{ 2|D^2 w|^2},$$
and
\begin{equation*}
  \big |\nabla |\nabla w|^2\big|^2 = 2\sum_{i=1}^n |\nabla w \cdot\nabla w_{x_i} |^2 \leq 4 |\nabla w|^2 \sum_{i=1}^n {|\nabla w_{x_i}|^2}= 4|\nabla w|^2 |D^2w|^2.
\end{equation*}
It then follows that
\begin{align}
    \label{aaux-new-eq3}
   - \int_{\Omega} \frac{|\nabla w|^{2p-4}}{w^{k-2}}  \Delta |\nabla w|^2 & \leq 2 \int_{\Omega} \frac{|\nabla w|^{2p-4}}{w^{k-2}} \nabla (\nu u+\lambda v) \cdot \nabla w\nonumber\\
   &\quad -2 \mu \int_{\Omega} \frac{|\nabla w|^{2p-2}}{w^{k-2}}-\frac{1}{2}  \int_{\Omega} \frac{|\nabla w|^{2p-6}}{w^{k-2}} \big |\nabla |\nabla w|^2\big|^2.
\end{align}
By \eqref{aaux-new-eq2} and \eqref{aaux-new-eq3}, we have
\begin{align}
\label{aaux-new-eq4}
    \Big(p-\frac{3}{2}\Big) \int_{\Omega} \frac{|\nabla w|^{2p-6}}{w^{k-2}}\big |\nabla |\nabla w|^2\big|^2 &\le  \int_{\partial \Omega} \frac{|\nabla w|^{2p-4}}{w^{k-2}} \frac{\partial |\nabla w|^2}{\partial \nu} + (k-2)\int_{\Omega} \frac{ |\nabla w|^{2p-4}}{w^{k-1}} \nabla w \cdot \nabla |\nabla w|^2 \nonumber\\
    & \quad + 2 \int_{\Omega} \frac{|\nabla w|^{2p-4}}{w^{k-2}} \nabla (\nu u+\lambda v) \cdot \nabla w-2 \mu \int_{\Omega} \frac{|\nabla w|^{2p-2}}{w^{k-2}}.
\end{align}
To get proper estimate for $\int_{\Omega} \frac{|\nabla w|^{2p-4}}{w^{k-1}} \nabla w \cdot \nabla |\nabla w|^2$, it is then essential to provide proper estimates for the integrals $  \int_{\partial \Omega} \frac{|\nabla w|^{2p-4}}{w^{k-2}} \frac{\partial |\nabla w|^2}{\partial \nu}$ and $ \int_{\Omega} \frac{|\nabla w|^{2p-4}}{w^{k-2}} \nabla (\nu u+\lambda v) \cdot \nabla w$.

\begin{lemma}
\label{lem-4-4}
For every $\epsilon>0$, there is $C>0$ such that
\begin{equation*}
     \int_{\partial \Omega} \frac{|\nabla w|^{2p-4}}{w^{k-2}} \frac{\partial |\nabla w|^2}{\partial \nu} \leq \epsilon \int_{ \Omega} \frac{|\nabla w|^{2p-6}}{w^{k-2}}\big |\nabla |\nabla w|^2\big|^2 + C \int_{ \Omega} \frac{|\nabla w|^{2p-2}}{w^{k-2}},
\end{equation*}
for all $t\in (0,T_{\max}(u_0,v_0))$.
\end{lemma}

\begin{proof}
It follows from the similar arguments in \cite[Lemma 4.4]{HKWS}.
\end{proof}

\begin{lemma}
\label{lem-4-3}
Let $p \ge  3$ and $k \ge p.$  For every given $\varepsilon>0,$ there is $M>0$ such that
\begin{equation*}
    \int_{\Omega}\frac{|\nabla w|^{2p-4}}{w^{k-2}} \nabla (\nu u+\lambda v)  \cdot \nabla w \leq M\int_{\Omega} \frac{(\nu u+\lambda v)^p}{w^{k-p}}  + {\varepsilon} \int_{\Omega}\frac{|\nabla w|^{2p-6}}{w^{k-2}}\big |\nabla |\nabla w|^2\big|^2 + {\varepsilon} \int_{\Omega}\frac{|\nabla w|^{2p}}{w^{k}}
    \end{equation*}
for all $t \in (0, T_{\rm max}(u_0,v_0))$.
\end{lemma}

\begin{proof}
 First,  we have that
\begin{align*}
    \int_{\Omega}\frac{|\nabla w|^{2p-4}}{w^{k-2}} \nabla (\nu u +\lambda v)\cdot \nabla w
  &= \underbrace{(k-2) \int_{\Omega}\frac{(\nu u+\lambda v) |\nabla w|^{2p-2}}{w^{k-1}}}_{I_{1}}  - \underbrace{\int_{\Omega}\frac{(\nu u+\lambda v) |\nabla w|^{2p-4}}{w^{k-2}}\Delta w}_{I_{2}} \\
&\,\,  - \underbrace{ (p-2) \int_{\Omega} \frac{(\nu u+\lambda v) |\nabla w|^{2p-6}}{w^{k-2}} \nabla w \cdot \nabla |\nabla w|^2}_{I_{3}}.
\end{align*}

Next, by Young's inequality, for every $B_1>0$, there exists a positive constant $A_1=A_1(k,p,B_1)>0$ such that
\begin{align}
\label{eq-4-31}
     I_1=(k-2) \int_{\Omega}\frac{(\nu u+\lambda v)|\nabla w|^{2p-2}}{w^{k-1}}&=(k-2) \int_{\Omega} \frac{\nu u+\lambda v}{w^{\frac{k-p}{p}}} \cdot  \frac{|\nabla w|^{2p-2}}{w^{\frac{k(p-1)}{p}}}\nonumber\\
     & \leq A_1 \int_{\Omega} \frac{(\nu u+\lambda v)^p}{w^{k-p}}+ B_1 \int_{\Omega} \frac{|\nabla w|^{2p}}{w^{k}}.
\end{align}

Now, by the fact that $\Delta w =\mu w-\nu u -\lambda v$, and Young's inequality, for every $B_2>0$, there exists a positive constant $A_2=A_2(k,p,\nu,\lambda, B_2)>0$ such that
\begin{align}
    \label{eq-4-32}
    {-I_2}&=-\int_{\Omega}\frac{(\nu u+\lambda v)|\nabla w|^{2p-4}}{w^{k-2}}\Delta w \nonumber\\
&=-\mu \int_{\Omega}\frac{(\nu u+\lambda v)|\nabla w|^{2p-4}}{w^{k-3}} +  \int_{\Omega}\frac{(\nu u+\lambda v)^2|\nabla w|^{2p-4}}{w^{k-2}}  \nonumber\\
    & \leq  \int_{\Omega} \frac{(\nu u+\lambda v)^2}{w^{\frac{2(k-p)}{p}}} \cdot \frac{|\nabla w|^{2p-4}}{w^{\frac{k(p-2)}{p}}}  \nonumber\\
    & \leq A_2 \int_{\Omega} \frac{(\nu u+\lambda v)^p}{w^{k-p}}+ B_2 \int_{\Omega} \frac{|\nabla w|^{2p}}{w^{k}}.
\end{align}

Finally, for every ${\varepsilon}>0$ and $B_3>0$,  there are positive constants $A_3=A_3({\varepsilon}, k,p,\nu,B_3)>0$ and
$A_4=A_4(A_4,B_3)$  such that
\begin{align}
\label{eq-4-33}
    {-I_3}&= -(p-2) \int_{\Omega} \frac{(\nu u+\lambda v) |\nabla w|^{2p-6}}{w^{k-2}} \nabla w \cdot \nabla |\nabla w|^2\nonumber\\
     & \leq  A_3 \int_{\Omega}\frac{(\nu u+\lambda v)^2|\nabla w|^{2p-4}}{w^{k-2}} + {{ \varepsilon}} \int_{\Omega}\frac{|\nabla w|^{2p-6}}{w^{k-2}}\big |\nabla |\nabla w|^2\big|^2 \nonumber\\
    & \leq A_4 \int_{\Omega} \frac{(\nu u+\lambda v)^p}{w^{k-p}}+ B_3 \int_{\Omega} \frac{|\nabla w|^{2p}}{w^{k}}+ {{\varepsilon}} \int_{\Omega}\frac{|\nabla w|^{2p-6}}{w^{k-2}}\big |\nabla |\nabla w|^2\big|^2.
\end{align}
Combining  \eqref{eq-4-31}, \eqref{eq-4-32} and \eqref{eq-4-33}  with $B_1=B_2=B_3=\frac{1}{3} \varepsilon$ and $M=M(\varepsilon,k,p,\nu):=A_1+A_2+A_4$ completes the proof.
\end{proof}

We now prove Proposition \ref{new-main-prop2}.

\begin{proof}[Proof of Proposition \ref{new-main-prop2}]
First, by \eqref{aaux-new-eq4} and Lemmas \ref{lem-4-4} and \ref{lem-4-3},  for any $0<\varepsilon<p-\frac{3}{2}$, there are
$M,C>0$ such that
\begin{align}
\label{aaux-new-eq5}
    \left(p-\frac{3}{2}\right) \int_{\Omega} \frac{|\nabla w|^{2p-6}}{w^{k-2}}\big |\nabla |\nabla w|^2\big|^2 &\le  \varepsilon \int_{\Omega} \frac{|\nabla w|^{2p-6}}{w^{k-2}}\big |\nabla |\nabla w|^2\big|^2
+\varepsilon \int_{\Omega} \frac{|\nabla w|^{2p}}{w^k}\nonumber\\
&\quad + (k-2)\int_{\Omega} \frac{ |\nabla w|^{2p-4}}{w^{k-1}} \nabla w \cdot \nabla |\nabla w|^2 \nonumber\\
&\quad + M\int_{\Omega} \frac{(\nu u+\lambda v)^p}{w^{k-p}}+C\int_\Omega \frac{|\nabla w|^{2p-2}}{w^{k-2}}.
\end{align}

Next,  by  \eqref{aaux-new-eq1} and \eqref{aaux-new-eq5}, we have
\begin{align}
\label{aaux-new-eq6}
& \frac{p-1}{p+k-3} \int_{\Omega} \frac{ |\nabla w|^{2p-4}}{w^{k-1}} \nabla w \cdot \nabla |\nabla w|^2\nonumber\\
 &\leq  \frac{\varepsilon}{p+k-3}\int_\Omega \frac{|\nabla w|^{2p}}{w^k}+\frac{M}{p+k-3}\int_\Omega \frac{(\nu u+\lambda v)^p}{w^{k-p}}\nonumber\\
& \,\, +\frac{C}{p+k-3}\int_\Omega \frac{|\nabla w|^{2p-2}}{w^{k-2}}  + \frac{p+k-3}{4\left(p-\frac{3}{2}-\varepsilon\right)} \int_{\Omega}\frac{|\nabla w|^{2p}}{w^{k}}.
\end{align}

Now, by \eqref{aaux-new-eq00},  and \eqref{aaux-new-eq6}, for any $0<\varepsilon<p-\frac{3}{2}$, there holds
\begin{align}
\label{aaux-new-eq7}
&\Big (k-1-\frac{(p+k-3)^2}{4(p-\frac{3}{2}-\varepsilon)} -\varepsilon \Big)\int_\Omega \frac{|\nabla w|^{2p}}{w^k}\nonumber\\
&\le  M\int_\Omega \frac{(\nu u+\lambda v)^p}{w^{k-p}}+(C+\mu)\int_\Omega \frac{|\nabla w|^{2p-2}}{w^{k-2}}.
\end{align}
By Young's inequality, for any $0<\varepsilon<p-\frac{3}{2}$, there is $\tilde M(\varepsilon,p)$ such that
\begin{equation}
\label{aaux-new-eq8}
(C+
\mu) \int_\Omega \frac{|\nabla w|^{2p-2}}{w^{k-2}}\le \varepsilon \int_\Omega\frac{|\nabla w|^{2p}}{w^k}+\tilde M\int_\Omega w^{2p-k}.
\end{equation}
By \eqref{aaux-new-eq7} and \eqref{aaux-new-eq8},
 \begin{align}
\label{aaux-new-eq9}
&\Big (k-1-\frac{(p+k-3)^2}{4(p-\frac{3}{2}-\varepsilon)} -2\varepsilon \Big)\int_\Omega \frac{|\nabla w|^{2p}}{w^k}\nonumber\\
&\le  M\int_\Omega \frac{(\nu u+\lambda v)^p}{w^{k-p}}+\tilde M \int_\Omega w^{2p-k}.
\end{align}

Finally,
since $p-\sqrt{2p-3} \leq k <  p+\sqrt{2p-3}$, one can find $\varepsilon>0$ such that
\begin{equation*}
    k-1-2\varepsilon- \frac{(p+k-3)^2}{4\left(p-\frac{3}{2}-\varepsilon\right)} >0.
\end{equation*}
The proposition then follows from \eqref{aaux-new-eq9}.
\end{proof}

\end{document}